\documentclass[11pt]{article}
\usepackage[top=2cm, bottom=2cm, left=2.5cm, right=2.5cm] {geometry}

\usepackage{comment}
\usepackage{amsmath,amssymb,theorem,color}
\numberwithin{equation}{section} 
\usepackage{amssymb}
\usepackage{color}
\usepackage{amsmath}
\usepackage{graphicx}
\usepackage{amsfonts}%
\newcommand{\R}{\ensuremath{\mathbb{R}}}

\newcommand{\N}{\ensuremath{\mathbb{N}}}

\newcommand{\K}{\mathbb{K}}
\newcommand{\B}{\mathbb{B}}
\newcommand{\cC}{\mathcal{C}}

\newcommand{\X}{\mathbb{X}}
\newcommand{\E}{\mathbb{E}}
\newcommand{\W}{\mathbb{W}}
\newcommand{\cD}{\mathcal{D}}
\newcommand{\cL}{\mathcal{L}}
\newcommand{\cM}{\mathcal{M}}
\newcommand{\bx}{\mathbf x}

\newcommand{\cB}{\mathcal{B}}

\newcommand{\ltn}{\ensuremath{\left| \! \left| \! \left|}}
\newcommand{\rtn}{\ensuremath{\right| \! \right| \! \right|}}

\newtheorem{theorem}{Theorem}[section]
{ \theorembodyfont{\normalfont} 

}

\newtheorem{definition}[theorem]{Definition}
\newtheorem{lemma}[theorem]{Lemma}
\newtheorem{corollary}[theorem]{Corollary}
\newtheorem{proposition}[theorem]{Proposition}

\newcounter{enumctr}
\newenvironment{enum}{\begin{list}{(\roman{enumctr})}{\usecounter{enumctr}}}{\end{list}}

\begin{document}

\title{Stability theory for Gaussian rough differential equations. Part I.}

\author{Luu Hoang Duc\\Max-Planck-Institut f\"ur Mathematik in den Naturwissenschaften, $\&$\\Institute of Mathematics, Vietnam Academy of Science and Technology\\ {\it E-mail: duc.luu@mis.mpg.de, lhduc@math.ac.vn}
}
\date{}
\maketitle

\begin{abstract}
We propose a quantitative direct method of proving the stability result for Gaussian rough differential equations in the sense of Gubinelli \cite{gubinelli}. Under the strongly dissipative assumption of the drift coefficient function, we prove that the trivial solution of the system under small noise is exponentially stable. 
\end{abstract}

{\bf Keywords:}
stochastic differential equations (SDE), Young integral, rough path theory, rough differential equations, exponential stability.


\section{Introduction}

This paper deal with the asymptotic stability criteria for rough differential equations of the form
\begin{equation}\label{RDE1}
dy_t = [A y_t + f(y_t)] dt + G(y_t)dx_t,
\end{equation}
or in the integral form
\begin{equation}\label{RDE2}
y_t = y_a + \int_a^t [A y_u + f(y_u) ]du + \int_a^t G(y_u) dx_u,\qquad t\in [a,T];
\end{equation}
where the nonlinear part $f: \R^d \to \R^d$ is globally Lipschitz function for simplicity and $G = (G_1,\ldots,G_m)$ is a collection of vector fields $G_j : \R^d \to \R^d$ such that
\begin{equation}\label{G}
G(y) = \begin{cases}
g(y), \text{where\ }  g = (g_1,\ldots, g_m), g_j \in C^{1+ \text{Lip}} \text{\ if\ } \nu \in (\frac{1}{2},1) \\
Cy, \quad \text{where} \quad C = (C_1,\ldots,C_m), C_j \in \R^{d\times d},\quad  \text{if}\quad \nu \in (\frac{1}{3},\frac{1}{2})\\
g(y), \quad \text{where} \quad g = (g_1,\ldots, g_m), g_j \in C^{3}_b(\R^d,\R^d), \quad  \text{if} \quad \nu \in (\frac{1}{3},\frac{1}{2}). 
\end{cases}
\end{equation}
Equation \eqref{RDE1} can be viewed as a controlled differential equation driven by rough path $x \in C^{\nu}([a,T],\R^m)$ for $\nu \in (\frac{1}{3},1]$, in the sense of Lyons \cite{lyons98}, \cite{lyonsetal07} where $x$ can also be considered as an element of the space $C^{p-{\rm var}}([a,T],\R^m)$ of finite $p$ - variation norm, with $p\nu \geq 1$. For instance, given $\bar{\nu} \in (\frac{1}{3},1]$, the path $x$ might be a realization of a $\R^m$-valued centered Gaussian process satisfying: there exists for any $T>0$ a constant $C_T$ such that for all $p \geq \frac{1}{\bar{\nu}}$
\begin{equation}\label{Gaussianexpect}
E \|X_t- X_s\|^{p} \leq C_T |t-s|^{p\bar{\nu}},\quad \forall s,t \in [0,T]. 
\end{equation}
By Kolmogorov theorem, for any $\nu \in (0,\bar{\nu})$ and any interval $[0,T]$ almost all realization of $X$ will be in $C^\nu([0,T])$. Such a stochastic process, in particular, can be a fractional Brownian motion $B^H$  \cite{mandelbrot} with Hurst exponent $H \in (\frac{1}{3},1)$, i.e. a family of $B^H = \{B^H_t\}_{t\in \R}$ with continuous sample paths and 
\[
E \|B^H_t- B^H_s\| = |t-s|^{2H}, \forall t,s \in \R.
\]
In this paper, we would like to approach system \eqref{RDE1}, where the second integral is well-understood as rough integral in the sense of Gubinelli \cite{gubinelli}. 
Such system satisfies the existence and uniqueness of solution given initial conditions, see e.g. \cite{gubinelli} or \cite{frizhairer} for a version without drift coefficient function, and \cite{riedelScheutzow} for a full version using $p$ - variation norms.

Notice that the question for global asymptotic dynamics of system \eqref{RDE1} is studied in \cite{hairer15}, \cite{hairer03}, \cite{hairer07}, \cite{hairer11}, \cite{hairer13}, under the general dissipativity condition for the drift coefficient function, in which they prove that there exists a unique smooth stationary density for \eqref{RDE1}, with convergence rate is either exponential or polynomial, depending on Hurst index $H$. \\
Meanwhile, the topic of asymptotic stability for pathwise solution of \eqref{RDE1} is studied in \cite{ducGANSch18} for which the noise is assumed to be fractional Brownian motion with small intensity. In addition, the local stability is studied in \cite{garrido-atienzaetal} and in \cite{GABSch18} for which the diffusion coefficient $g(x)$ is rather flat, i.e. $g(0) = D_yg(0) = 0$ for the Young differential equations and $g(0) = D_yg(0) = D_{yy}g(0)= 0$ for the rough differential equations. In all mentioned references, the technique in use is semigroup technique together with the help of fractional calculus.\\
To study the local stability, we impose conditions for matrices $A\in \R^{d\times d}$ such that $A$ is negative definite, i.e. there exists a $\lambda_A >0$ such that
\begin{equation}\label{lambda}
\langle y, Ay \rangle \leq - \lambda_A \|y\|^2.
\end{equation}
The strong condition \eqref{lambda} is still able to cover interesting cases, for instance all matrices with negative real part eigenvalues, under a transformation, since there exists a positive definite matrix $\mathcal{Q}$ which is the solution of the matrix equation
\[
A^{\rm T} \mathcal{Q} + \mathcal{Q} A = D
\]
where $D$ is a symmetric negative definite matrix \cite[Chapter 2 \& Chapter 5]{burton}.\\
To study the local stability, we will assume that the nonlinear part $f: \R^d \to \R^d$ is locally Lipschitz function such that 
\begin{equation}\label{condf}
f(0) =0\quad \text{and} \quad \|f(y)\| \leq \|y\| h(\|y\|)
\end{equation}
where $h: \R^+ \to \R^+$ is an increasing function which is bounded above by a constant $C_f$. Our assumption is somehow still global, but it has an advantage of being able to treat the local dynamics as well. We refer to \cite{garrido-atienzaetal} and \cite{GABSch18} for real local versions on a small neighborhood $B(0,\rho)$ of the trivial solution, using the cutoff technique.\\  
In this paper, we also assume that $g(0) = 0$ and $g \in C^{3}_b$ in case $\nu\in (\frac{1}{3},\frac{1}{2})$ with bounded derivatives $C_g$ (which also include the Lipschit coefficient of the highest derivative). System \eqref{RDE1} then admits an equilibrium which is the trivial solution. Our main stability results are then formulated as follows.

\begin{theorem}[Stability for Young systems]\label{globalstab}
	Assume $X_\cdot(\omega)$ is a Gaussian process satisfying \eqref{Gaussianexpect}, and $\bar{\nu}>\nu >\frac{1}{2}$ is fixed. Assume further that conditions \eqref{lambda}, \eqref{condf} are satisfied, where $\lambda_A > h(0)$. Then there exists an $\epsilon >0$ such that given $C_g < \epsilon$, and for almost sure all realizations $x_\cdot=X_\cdot(\omega)$, the zero solution of \eqref{RDE1} is locally exponentially stable. If in addition $\lambda_A > C_f$, then we can choose $\epsilon$ so that the zero solution of \eqref{RDE1} is globally exponentially stable a.s. 
\end{theorem}

\begin{theorem}[Stability for rough systems]\label{stablinRDE}
	Assume $X_\cdot(\omega)$ is a Gaussian process satisfying \eqref{Gaussianexpect}, and $\frac{1}{2}>\bar{\nu}>\nu >\frac{1}{3}$ is fixed. Assume further that $G(y) = Cy$ and conditions \eqref{lambda}, \eqref{condf} are satisfied, where $\lambda_A > h(0)$. Then the conclusion of Theorem \ref{globalstab} on local stability of the zero solution holds for almost sure all realizations $x$ of $X$. If in addition $\lambda_A > C_f$, then we can choose $\epsilon$ so that the zero solution of \eqref{RDE1} is globally exponentially stable a.s. 
\end{theorem}
Our method follows the direct method of Lyapunov, which aims to estimate the norm growth (or a Lyapunov-type function) of the solution in discrete intervals using the rough estimates for the angular equation which is feasible thanks to the change of variable formula for rough integral defined in the sense of Gubinelli. It is then sufficient to study the local and global exponential stablity of the corresponding random differential inequality, which can be done with random norm techniques in \cite{arnold}. We show in Part I that our method works for Young equations or for rough systems in which $G(y)= Cy$, since it is not necessary to prove the integrability of $\ltn \theta, \theta^\prime \rtn_{x,2\alpha,[a,b]}$ in order to get the pathwise stability. \\
Part II \cite{duc19part1} is to present the result for the case $G(y) = g(y) \in C^3_b$, for which a necessary assumption is the integrability of solution. This assumption is straightforward for Young equations but not trivial for the rough case, and even difficult to prove under the H\"older norm. Specifically, the concept of {\it greedy times} for H\"older norms and similar result to \cite[Theorem 6.3]{cassetal} on the main tail estimate of the number of greedy time under the $\alpha$-H\"older norm is not easy to prove. Fortunately, we can overcome this issue by studying Gubinelli approach under the modified $(p-\sigma)$ - variation seminorms in order to apply \cite[Theorem 6.3]{cassetal} directly.\\
We close the introduction part with a note that our method still works for the case $\nu \in (\frac{1}{4},\frac{1}{3}]$ with an extension of Gubinelli derivative to the second order, although the computation would be rather complicated. Moreover, it could also be applied for proving the general case in which $g$ is unbounded, even though we then need to prove the existence and uniqueness theorem and also the integrability of the solution. The reader is referred to \cite{lejay} and \cite{coutinlejay} for this approach, in which the differential equation is understood in the sense of Davie \cite{davie}.

\section{Rough differential equations}
\subsection{$\nu \in (\frac{1}{2},1)$: Young differential equations}

We would like to give a brief introduction to Young integrals. Given any compact time interval $I \subset \R$, let $C(I,\R^d)$ denote the space of all continuous paths $y:\;I \to \R^d$ equipped with sup norm $\|\cdot\|_{\infty,I}$ given by $\|y\|_{\infty,I}=\sup_{t\in I} \|y_t\|$, where $\|\cdot\|$ is the Euclidean norm in $\R^d$. We write $y_{s,t}:= y_t-y_s$. For $p\geq 1$, denote by $\cC^{p{\rm-var}}(I,\R^d)\subset C(I,\R^d)$ the space of all continuous path $y:I \to \R^d$ which is of finite $p$-variation 
\begin{eqnarray}
\ltn y\rtn_{p\text{-var},I} :=\left(\sup_{\Pi(I)}\sum_{i=1}^n \|y_{t_i,t_{i+1}}\|^p\right)^{1/p} < \infty,
\end{eqnarray}
where the supremum is taken over the whole class of finite partition of $I$. $\cC^{p{\rm-var}}(I,\R^d)$ equipped with the $p-$var norm
\begin{eqnarray*}
	\|y\|_{p\text{-var},I}&:=& \|y_{\min{I}}\|+\ltn y\rtn_{p\rm{-var},I},
\end{eqnarray*}
is a nonseparable Banach space \cite[Theorem 5.25, p.\ 92]{friz}. Also for each $0<\alpha<1$, we denote by $C^{\alpha}(I,\R^d)$ the space of H\"older continuous functions with exponent $\alpha$ on $I$ equipped with the norm
\[
\|y\|_{\alpha,I}: = \|y_{\min{I}}\| + \ltn y\rtn_{\alpha,I}=\|y(a)\| + \sup_{s<t\in I }\frac{\|y_{s,t}\|}{(t-s)^\alpha},
\]
A continuous map $\overline{\omega}: \Delta^2(I)\longrightarrow \R^+, \Delta^2(I):=\{(s,t): \min{I}\leq s\leq t\leq \max{I}\}$ is called a {\it control} if it is zero on the diagonal and superadditive, i.e.  $\overline{\omega}_{t,t}=0$ for all $t\in I$, and  $\overline{\omega}_{s,u}+\overline{\omega}_{u,t}\leq \overline{\omega}_{s,t}$ for all $s\leq u\leq t$ in $I$.\\
Now, consider $y\in \cC^{q{\rm-var}}(I,\cL(R^m,\R^d))$ and $x\in \cC^{p{\rm -var}}(I,\R^m)$ with  $\frac{1}{p}+\frac{1}{q}  > 1$, the Young integral $\int_I y_t d x_t$ can be defined as 
\[
\int_I y_s d x_s:= \lim \limits_{|\Pi| \to 0} \sum_{[u,v] \in \Pi} y_u x_{u,v} ,
\]
where the limit is taken on all the finite partition $\Pi=\{ \min{I} =t_0<t_1<\cdots < t_n=\max{I} \}$ of $I$ with $|\Pi| := \displaystyle\max_{[u,v]\in \Pi} |v-u|$ (see \cite[p.\ 264--265]{young}). This integral satisfies additive property by the construction, and the so-called Young-Loeve estimate \cite[Theorem 6.8, p.\ 116]{friz}
\begin{eqnarray}\label{YL0}
\Big\|\int_s^t y_u d x_u-y_s x_{s,t}\Big\| &\leq& K(p,q) \ltn y\rtn_{q\text{-var},[s,t]} \ltn x\rtn_{p\text{-var},[s,t]} \notag\\
&\leq& K(p,q) |t-s|^{\frac{1}{p}+\frac{1}{q}} \ltn y \rtn_{\frac{1}{p},[s,t]} \ltn x \rtn_{\frac{1}{q}{\rm -Hol},[s,t]}, 
\end{eqnarray}
for all $[s,t]\subset I$, where 
\begin{equation}\label{constK}
K(p,q):=(1-2^{1-\frac{1}{p} - \frac{1}{q}})^{-1}.
\end{equation}

\begin{theorem}[Existence, uniqueness and integrability of the solution]
	Under assumptions (${\textbf H}_1$), (${\textbf H}_2$), there exists a unique solution of equation \eqref{RDE1} on any interval $[a,b]$. Moreover $\ltn y \rtn_{q-{\rm var},[a,b]}$ is integrable.
\end{theorem}
\begin{proof}
	Since $\nu >\frac{1}{2}$, \eqref{RDE1} is a Young equation which satisfies the assumptions of Theorem 3.6 and Theorem 4.4 in \cite{congduchong17} on the existence and uniqueness of solution for \eqref{RDE1} and its backward equation. Moreover to estimate $\ltn x \rtn_{q-{\rm var},[a,b]}$, we apply \cite[Lemma 3.3]{congduchong17} to conclude that there exists a function 
	\[
	F(\ltn x \rtn_{p-{\rm var},[a,b]}) = 4^p(\log 2) \max\{\|A\| + C_f, (K+1)C_g \}  \Big[(b-a)^p + \ltn x \rtn_{p-{\rm var},[a,b]}^p\Big]
	\]
	such that
	\begin{eqnarray}\label{xqvar} 
	\ltn y \rtn_{q-{\rm var},[a,b]} &\leq& \|y_a\| \exp \Big\{F(\ltn x \rtn_{p-{\rm var},[a,b]}) \Big\} \notag\\
	\|y\|_{\infty,[a,b]} \leq\ltn y \rtn_{q-{\rm var},[a,b]} + \|y_a\|&\leq& \|y_a\| \Big(1+ \exp \Big\{F(\ltn x \rtn_{p-{\rm var},[a,b]}) \Big\} \Big). 
	\end{eqnarray}
	From \cite{nualart} (see also \cite[Proposition 2.1,p.18]{Hu}) the random variable $Z:= e^{ \ltn x\rtn_{p-{\rm var},[0,1]}}$, with $1<p<2$, has finite moments of any order, provided that $x$ is a realization of Gaussian stochastic process. That proves the integrability of $\ltn y \rtn_{q-{\rm var},[a,b]}$ and $	\|y\|_{\infty,[a,b]} $. Notice that the integrability of $\ltn y \rtn_{q-{\rm var},[a,b]}$ and $	\|y\|_{\infty,[a,b]} $ can also be proved using \cite[Theorem 6.3]{cassetal} with better estimates. 
\end{proof}

\subsection{$\nu \in (\frac{1}{3},\frac{1}{2})$: controlled differential equations}

We also introduce the construction of the integral using rough paths for the case $y,x \in C^\alpha(I)$ when $\alpha\in(\frac{1}{3},\nu)$. To do that, we need to introduce the concept of rough paths. Following \cite{frizhairer}, a couple $\bx=(x,\X)$, with $x \in C^\alpha(I,\R^m)$ and $\X \in C^{2\alpha}_2(\Delta^2(I),\R^m \otimes \R^m):= \{\X: \sup_{s<t} \frac{\|\X_{s,t}\|}{|t-s|^{2\alpha}} < \infty \}$ where the tensor product $\R^m \otimes \R^n$ can be indentified with the matrix space $\R^{m\times n}$, is called a {\it rough path} if they satisfies Chen's relation
\begin{equation}\label{chen}
\X_{s,t} - \X_{s,u} - \X_{u,t} = x_{u,t} \otimes x_{s,u},\qquad \forall \min{I} \leq s \leq u \leq t \leq \max{I}. 
\end{equation}
$\X$ is viewed as {\it postulating} the value of the quantity $\int_s^t x_{s,r} \otimes dx_r := \X_{s,t}$ where the right hand side is taken as a definition for the left hand side. Denote by $\cC^\alpha(I) \subset C^\alpha \oplus C^{2\alpha}_2$ the set of all rough paths in $I$, then $\cC^\alpha$ is a closed set but not a linear space, equipped with the rough path semi-norm 
\begin{equation}\label{translated}
\ltn \bx \rtn_{\alpha,I} := \ltn x \rtn_{\alpha,I} + \ltn \X \rtn_{2\alpha,\Delta^2(I)}^{\frac{1}{2}} < \infty.  
\end{equation}
Let $3>p> 2, \nu \geq \frac{1}{p}$. Throughout this paper, we will assume that  $x(\omega): I \to \R^m$ and $\X(\omega): I \times I \to \R^m \otimes \R^m$ are random funtions that satisfy Chen's relation relation \eqref{chen} and
\begin{equation}\label{expect}
\Big(E \|x_{s,t} \|^p\Big)^{\frac{1}{p}} \leq C |t-s|^\nu, \quad \text{and} \quad 	\Big(E \|\X_{s,t}\|^{\frac{p}{2}}\Big)^{\frac{2}{p}} \leq C |t-s|^{2\nu},\forall s,t \in I
\end{equation}
for some constant $C$. Then, due to the Kolmogorov criterion for rough paths \cite[Appendix A.3]{friz} for all $\alpha \in (\frac{1}{3},\nu)$ there is a version of $\omega-$wise $(x,\X)$ and random variables $K_\alpha \in L^p, \K_\alpha \in L^{\frac{p}{2}}$, such that, $\omega-$wise speaking, for all $s,t \in I$,
\[
\|x_{s,t}\| \leq K_\alpha |t-s|^\alpha, \quad \|\X_{s,t}\| \leq \K_\alpha |t-s|^{2\alpha}.
\] 
In particular, if $\beta -\frac{1}{q} > \frac{1}{3}$ then, for every $\alpha \in (\frac{1}{3},\beta-\frac{1}{q})$ we have $(x,\X) \in \cC^\alpha.$ Moreover, we could choose $\alpha$ abit smaller such that $x \in C^{0,\alpha}(I):= \{x \in C^\alpha: \lim \limits_{\delta \to 0}\sup_{0<t-s <\delta} \frac{\|x_{s,t}\|}{|t-s|^\alpha} = 0\}$ and $\X \in C^{0,2\alpha}(\Delta^2(I)):= \{ \X \in C^{2\alpha}(\Delta^2(I)): \lim \limits_{\delta \to 0}\sup_{0<t-s <\delta} \frac{\|\X_{s,t}\|}{|t-s|^{2\alpha}} = 0  \}$, then $\cC^{\alpha}(I) \subset C^{0,\alpha}(I) \oplus C^{0,2\alpha}(\Delta^2(I))$ is separable due to the separability of $C^{0,\alpha}(I)$ and $C^{0,2\alpha}(\Delta^2(I))$.

\subsubsection{Controlled rough paths}

A path $y \in C^\alpha(I,\cL(\R^m,\R^d))$ is then called to be {\it controlled by} $x \in C^\alpha(I,\R^m)$ if there exists a tube $(y^\prime,R^y)$ with $y^\prime \in C^\alpha(I,\cL(\R^m,\cL(\R^m,\R^d))), R^y \in C^{2\alpha}(\Delta^2(I),\cL(\R^m,\R^d))$ such that
\[
y_{s,t} = y^\prime_s x_{s,t} + R^y_{s,t},\qquad \forall \min{I}\leq s \leq t \leq \max{I}.
\]
$y^\prime$ is called Gubinelli derivative of $y$, which is uniquely defined as long as $x \in C^\alpha\setminus C^{2\alpha}$ (see \cite[Proposition 6.4]{frizhairer}). The space $\cD^{2\alpha}_x(I)$ of all the couple $(y,y^\prime)$ that is controlled by $x$ will be a Banach space equipped with the norm
\begin{eqnarray*}
	\|y,y^\prime\|_{x,2\alpha,I} &:=& \|y_{\min{I}}\| + \|y^\prime_{\min{I}}\| + \ltn y,y^\prime \rtn_{x,2\alpha,I},\qquad \text{where} \\
	\ltn y,y^\prime \rtn_{x,2\alpha,I} &:=& \ltn y^\prime \rtn_{\alpha,I} +   \ltn R^y\rtn_{2\alpha,I},
\end{eqnarray*}
where we omit the value space for simplicity of presentation. Now fix a rough path $(x,\W)$, then for any $(y,y^\prime) \in \cD^{2\alpha}_x (I)$, it can be proved that the function $F \in C^\alpha(\Delta^2 (I),\R^d)$ defined by 
\[
F_{s,t} := y_s x_{s,t} + y^\prime_s \X_{s,t}
\] 
belongs to the space 
\begin{eqnarray*}
	C^{\alpha, 3\alpha}_2(I) &:=& \Big \{ F\in C^\alpha(\Delta^2(I)): F_{t,t} =0 \quad \text{and}\\ && \quad \qquad \qquad \qquad \qquad \ltn \delta F \rtn_{3\alpha,I} := \sup_{\min{I} \leq s \leq u \leq t \leq \max{I}} \frac{\|F_{s,t} - F_{s,u}-F_{u,t}\|}{|t-s|^{3\alpha}} < \infty \Big\}.
\end{eqnarray*}
Thanks to the sewing lemma \cite[Lemma 4.2]{frizhairer}, the integral $\int_s^t y_u dx_u$ can be defined as  
\[
\int_s^t y_u dx_u := \lim \limits_{|\Pi| \to 0} \sum_{[u,v] \in \Pi} [ y_{u}x_{u,v} + y^\prime_u \X_{u,v} ]
\]
where the limit is taken on all the finite partition $\Pi$ of $I$ with $|\Pi| := \displaystyle\max_{[u,v]\in \Pi} |v-u|$ (see \cite{gubinelli}). Moreover, there exists a constant $C_\alpha = C_{\alpha,|I|} >1$ with $|I| := \max{I} - \min{I}$, such that
\begin{equation}\label{roughEst}
\Big\|\int_s^t y_u dx_u - y_s x_{s,t} + y^\prime_s \X_{s,t}\Big\| \leq C_\alpha |t-s|^{3\alpha} \Big(\ltn x \rtn_{\alpha,[s,t]} \ltn R^y \rtn_{2\alpha,\Delta^2[s,t]} + \ltn y^\prime\rtn_{\alpha,[s,t]} \ltn \X \rtn_{2\alpha,\Delta^2[s,t]}\Big).
\end{equation}
From now on, if no other emphasis, we will simply write $\ltn x \rtn_{\alpha}$ or $\ltn \X \rtn_{2\alpha}$ without addressing the domain in $I$ or $\Delta^2(I)$. In particular, for any $f \in C^2_b(\R^d,\R^d)$, then $f(x) \in \cD^{2\alpha}_x$ with $f(x)^\prime = \nabla f(x)$ and 
\[
\ltn f(x), \nabla f(x) \rtn_{x, 2\alpha} \leq \|\nabla^2\|_\infty \big(\ltn x \rtn_\alpha + \frac{1}{2}\ltn x \rtn_\alpha^2 \big).
\]
In that case \eqref{roughEst} becomes
\[
\Big|\int_s^t f(x_u) dx_u - f(x_s)x_{s,t} + \nabla f(x_s) \X_{s,t}\Big| \leq C |t-s|^{3\alpha} \|f\|_{C^2_b} \Big( \ltn x \rtn_\alpha^3+ \ltn x\rtn_\alpha \ltn \X \rtn_{2\alpha}\Big).
\]
Moreover, in case $f \in C^3_b$ then we get the formula for integration by composition
\[
f(x_t) = f(x_s) + \int_s^t \nabla f(x_u) dx_u + \frac{1}{2} \int_s^t \nabla^2f(x_u) d[x]_{s,u},
\]
where the last integral is understood in the Young sense and $[x]_{s,t}:= x_{s,t} \otimes x_{s,t}  - 2 \text{\ Sym\ } (\X_{s,t}) \in C^{2\alpha}$. Notice that for geometric rough path $\X_{s,t} = \int_s^t x_{s,r} \otimes dx_r$, then $\text{\ Sym\ } (\X_{s,t}) = \frac{1}{2} x_{s,t} \otimes x_{s,t}$, thus $[x]_{s,t} \equiv 0.$

\begin{lemma}[Change of variables formula]
	Assume that $\alpha > \frac{1}{3}$, $V \in C^3_b(\R^d,\R)$ and $y \in C^{\alpha}(I,\R) $ is a solution of the rough differential equation
	\begin{equation}\label{roughde1}
	y_t = y_s + \int_s^t f(y_u)du + \int_s^t g(y_u)dx_u,\quad \forall \min{I} \leq s \leq t \leq \max{I}.
	\end{equation}
	Then one get the change of variable formula
	\begin{eqnarray}\label{Roughformula}
	V(y_t) &=& V(y_s) + \int_s^t \langle D_y V(y_u), f(y_u)\rangle  du + \int_s^t \langle D_y V(y_u) g(y_u) \rangle d x_u \notag \\
	&& + \frac{1}{2} \int_s^t D_{yy}V(y_u) [g(y_u),g(y_u)] d[x]_{s,u},  
	\end{eqnarray}
	where
	\[
	[ D_y V(y) g(y)]^\prime_s = \langle D_y V(y_s), D_y g(y_s)g(y_s)\rangle + D_{yy}V(y_s)[g(y_s),g(y_s)].
	\]
\end{lemma}
\begin{proof}
	Using the Taylor expansion, it is easy to see that
	\[
	V(y_t) = V(y_s) + \langle D_y V(y_s),y_{s,t}\rangle + \frac{1}{2} D_{yy} V(y_s)[y_{s,t},y_{s,t}] +  O(|t-s|^{3\alpha}).
	\]
	On the other hand, it follows from \eqref{roughde1} and \eqref{roughEst} that
	\begin{eqnarray*}
		y_{s,t} &=& f(y_s)(t-s) + g(y_s) x_{s,t} + [g(y)]^\prime_s \X_{s,t} + O(|t-s|^{3\alpha}) \\
		&=& f(y_s)(t-s) + g(y_s) x_{s,t} + D_y g(y_s) g(y_s) \X_{s,t} + O(|t-s|^{3\alpha}).
	\end{eqnarray*}
	As the result,
	\begin{eqnarray*}
		V(y)_{s,t}&=& \langle D_y V(y_s), f(y_s) \rangle (t-s) + \langle D_y V(y_s), g(y_s) \rangle x_{s,t} + D_y V(y_s) D_y g(y_s)g(y_s) \X_{s,t} \\
		&&+ \frac{1}{2} D_{yy}V(y_s)[g(y_s),g(y_s)] x_{s,t} \otimes x_{s,t} + O(|t-s|^{3\alpha}) \\
		&=&\langle D_y V(y_s), f(y_s)\rangle (t-s) + \langle D_y V(y_s), g(y_s) \rangle x_{s,t}+ \frac{1}{2} D_{yy}V(y_s)[g(y_s),g(y_s)] [x]_{s,t}  \\
		&&+ \Big(D_y V(y_s) D_y g(y_s)g(y_s) + D_{yy}V(y_s)[g(y_s),g(y_s)]\Big) \X_{s,t} + O(|t-s|^{3\alpha}) \\
		&=& \langle D_y V(y_s), f(y_s)\rangle (t-s) + \langle D_y V(y_s), g(y_s) \rangle x_{s,t} + [ D_y V(y) g(y)]^\prime_s \X_{s,t} \\
		&&+ \frac{1}{2} D_{yy}V(y_s)[g(y_s),g(y_s)] [x]_{s,t} + O(|t-s|^{3\alpha}),
	\end{eqnarray*}
	which is the discretization version of \eqref{Roughformula}. The conclusion is then a direct consequence of the sewing lemma.\end{proof}

\subsubsection{Greedy times}
For any $\nu \in (\frac{1}{3},\frac{1}{2})$ and on each compact interval $I$ such that $|I|=\max{I} - \min{I} =1$, consider a rough path $\bx = (x,\X) \in C^\nu(I)$ with H\"older norm. Then given $\alpha \in (\frac{1}{3},\nu)$, we construct for any fixed $\gamma \in (0,1)$ the sequence of greedy times $\{\tau_i(\gamma,I,\alpha)\}_{i \in \N}$ w.r.t. H\"older norms 
\begin{equation}\label{greedytime}
\tau_0 = \min{I},\quad \tau_{i+1}:= \inf\Big\{t>\tau_i:  \ltn \bx \rtn_{\alpha, [\tau_i,t]} = \gamma \Big\}\wedge \max{I}.
\end{equation}
Denote by $N_{\gamma,I,\alpha}(\bx):=\sup \{i \in \N: \tau_i \leq \max{I}\}$. From the definition \eqref{greedytime}, it follows that 
\[
\gamma < (\tau_{i+1} - \tau_i)^{\nu-\alpha} \Big(\ltn x\rtn_{\nu,I} + \ltn \X \rtn_{2\nu, \Delta^2(I)}^{\frac{1}{2}}\Big), 
\]
which implies that 
\begin{equation*}
	|I| \geq\tau_{N_{\gamma,I,\alpha}(\bx)} - \min{I} = \sum_{i=0}^{N_{\gamma,I,\alpha}(\bx)-1} (\tau_{i+1} - \tau_i) \geq N_{\gamma,I,\alpha}(\bx)\gamma^{\frac{1}{\nu-\alpha}}\Big(\ltn x\rtn_{\nu,I} + \ltn \X \rtn_{2\nu,\Delta^2(I)}^{\frac{1}{2}}\Big)^{\frac{-1}{\nu-\alpha}}.
\end{equation*}
This proves that
\begin{equation}
N_{\gamma,I,\alpha}(\bx) \leq |I|\gamma^{\frac{-1}{\nu-\alpha}}\Big(\ltn x\rtn_{\nu,I} +\ltn \X \rtn_{2\nu,\Delta^2(I)}^{\frac{1}{2}}\Big)^{\frac{1}{\nu-\alpha}}.
\end{equation}
Also, we construct another sequence of greedy time $\{\bar{\tau}_i(\gamma,I,\alpha)\}_{i \in \N}$ given by
\begin{equation}\label{greedywitht}
\bar{\tau}_0 = \min{I},\quad \bar{\tau}_{i+1}:= \inf\Big\{t>\bar{\tau}_i:  (t-\bar{\tau}_i)^{1-2\alpha} + \ltn \bx \rtn_{\alpha, [\bar{\tau}_i,t]} = \gamma \Big\}\wedge \max{I},
\end{equation}
and denote by $\bar{N}_{\gamma,I,\alpha}(\bx):=\sup \{i \in \N: \bar{\tau}_i \leq \max{I}\}$. Then on any interval $J$ such that $|J| = \Big(\frac{\gamma}{2}\Big)^{\frac{1}{1-2\alpha}}$ and with the sequence $\{\tau_i(\frac{\gamma}{2},J,\alpha)\}_{i\in\N}$ it follows that
\[
(\tau_{i+1} - \tau_i)^{1-2\alpha} + \ltn \bx \rtn_{\alpha, [\tau_i,\tau_{i+1}]} \leq \frac{\gamma}{2} + \frac{\gamma}{2} = \gamma, 
\]
hence there is a most one greedy time of the sequence $\bar{\tau}_i$ lying in each interval $[\tau_i(\frac{\gamma}{2},J,\alpha), \tau_{i+1}(\frac{\gamma}{2},J,\alpha)]$. That being said, if we divide $I$ into sub-interval $J_k$ of length $|J_k| \equiv |J|= \Big(\frac{\gamma}{2}\Big)^{\frac{1}{1-2\alpha}}$, then it follows that
\begin{equation}
\bar{N}_{\gamma,I,\alpha}(\bx) \leq \sum_{k=1}^{m} N_{\frac{\gamma}{2},J_k,\alpha}(\bx),\quad m := \Big\lceil \frac{|I|}{|J|} \Big\rceil.
\end{equation}

\begin{theorem}[Existence and uniqueness of the solution]
	Assume that $G(y)= Cy$, there exists a unique solution of equation \eqref{RDE1} and also of the backward equation on any interval $[a,b]$. 
\end{theorem}
\begin{proof}	
	To make our presentation self contained, we give a direct proof here for the rough differential equation
	\[
	dy = [A y_t + f(y_t)]dt + C y_t dx_t = F(y_t) dt + C y_t dx_t,
	\]
	or in the integral form
	\begin{equation}\label{linRDE3}
	y_t = G(y,y^\prime)_t = y_a + \int_a^t F(y_u) du + \int_a^t C y_u dx_u,\qquad t\in [a,T],
	\end{equation}
	where $F(0)$ is globally Lipschitz continuous with Lipschitz coefficient $L_f = \|A\| + C_f$. Denote by $\cD^{2\alpha}_x(y_a,C y_a)$  the set of paths $(y,y^\prime)$ controlled by $x$ in $[a,T]$ with $y_a$ and $y^\prime_a = Cy_a$ fixed. Consider the mapping defined by
	\[
	\cM: \cD^{2\alpha}_x(y_a,C y_a) \to \cD^{2\alpha}_x(y_a,Cy_a),\qquad \cM (y,y^\prime)_t := (G(y,y^\prime)_t, Cy_t). 
	\]
	Then similar to \cite{gubinelli} we are going to estimate $\ltn \cM (y,y^\prime)\rtn_{x,2\alpha} = \ltn C y \rtn_\alpha + \ltn R^{F(y,y^\prime)} \rtn_{2\alpha}$ using $\ltn(y,y^\prime)\rtn_{x,2\alpha} = \ltn y^\prime \rtn_\alpha + \ltn R^y \rtn_{2\alpha}$. Since
	\begin{eqnarray*}
	\ltn C y \rtn_\alpha &\leq& \|C\| \ltn y \rtn_\alpha \leq  \|C\| \Big(\|y^\prime\|_{\infty} \ltn x\rtn_\alpha + (T-a)^\alpha \ltn R^y \rtn_{2\alpha}\Big)\\
	&\leq& \|C\| \ltn x \rtn_\alpha \|y^\prime_a\| + \|C\|(T-a)^\alpha \ltn x \rtn_\alpha \ltn y^\prime \rtn_\alpha + \|C\| (T-a)^\alpha \ltn R^y \rtn_{2\alpha}
	\end{eqnarray*}
	and
	\begin{eqnarray*}
		\|R^{F(y,y^\prime)}_{s,t}\| &\leq&  \Big\|\int_s^t F(y_u) du\Big\| + \Big\|\int_s^t C y_u dx_u - Cy_s x_{s,t}\Big\| \\
		&\leq&  L_f|t-s| \|y\|_{\infty,[s,t]}+ \|C\| \|y^\prime\|_{\infty,[s,t]} |\X_{s,t}| \\
		&&+ C_\alpha |t-s|^{3\alpha} \Big[\ltn x \rtn_{\alpha,[s,t]} \|C\| \ltn R^y \rtn_{2\alpha,[s,t]} + \|C\| \ltn y^\prime \rtn_{\alpha,[s,t]} \ltn \X \rtn_{2\alpha,\Delta^2([s,t])} \Big],
	\end{eqnarray*}
	where we can choose $T-a <1$ so that $C_\alpha$ can be bounded from above by $C_{\alpha}(1)$. In addition 
	\[
	\|y\|_{\infty,[s,t]} \leq \|y_a\| + \|y^\prime_a\| (T-a)^{\alpha} \ltn x \rtn_{\alpha} + (T-a)^{2 \alpha} \ltn R^y \rtn_{2 \alpha},
	\]
	thus it follows that
	\begin{eqnarray*}
		&&\ltn R^{F(y,y^\prime)} \rtn_{2\alpha} \\
		&\leq& (T-a)^{1-2\alpha} L_f \|y_a\| + (T-a)^{1-\alpha} L_f \ltn x \rtn_\alpha \|y_a^\prime\| + L_f(T-a)\ltn R^y \rtn_{2\alpha} \\
		&&+ \|C\| \ltn \X \rtn_{2\alpha} (|y_a^\prime|+ (T-a)^\alpha \ltn y^\prime \rtn_\alpha)+ C_\alpha \|C\| (T-a)^\alpha \Big[\ltn x \rtn_{\alpha} \ltn R^y \rtn_{2\alpha} + \ltn y^\prime \rtn_{\alpha} \ltn \X \rtn_{2\alpha} \Big] 
	\end{eqnarray*}
	All in all, we can estimate $\ltn \cM (y,y^\prime)\rtn_{x,2\alpha}$ as follows
	\allowdisplaybreaks
	\begin{eqnarray*}
		&&\ltn \cM(y,y^\prime) \rtn_{x, 2\alpha} \\
		&\leq& \|C\|\Big[(\|y^\prime_a\| + (T-a)^\alpha \ltn y^\prime \rtn_\alpha)\ltn x \rtn_\alpha + (T-a)^\alpha \ltn R^y \rtn_{2\alpha} \Big] + \ltn R^{F(y,y^\prime)} \rtn_{2\alpha}\\
		&\leq& (T-a)^{1-2\alpha}L_f \|y_a\|  + \Big[\Big(\|C\| + (T-a)^{1-\alpha} L_f \Big) \ltn x \rtn_{\alpha}+ \|C\| \ltn \X \rtn_{2\alpha} \Big] \|C y_a\| \\
		&&+ \Big[(T-a)^\alpha \|C\| \ltn x \rtn_\alpha + (T-a)^\alpha \|C\| (1+ C_\alpha) \ltn \X \rtn_{2\alpha} \Big] \ltn y^\prime \rtn_\alpha \\
		&&+ \Big[\|C\| (T-a)^\alpha + (T-a)L_f + C_\alpha \|C\| (T-a)^\alpha \ltn x \rtn_\alpha \Big] \ltn R^y \rtn_{2 \alpha}\\
		&\leq& \Big(L_f + \|C\| + \|C\| C_\alpha\Big) (1 + \|C\|) \mu \|y_a\|+ \Big[L_f + \|C\| + \|C\| C_\alpha \Big]\mu \Big(\ltn y^\prime \rtn_\alpha + \ltn R^y \rtn_{2\alpha}\Big)\\
		&\leq& \mu \Big(\|y_a\|+\|Cy_a\| +\ltn y,y^\prime\rtn_{x,2\alpha}\Big)
	\end{eqnarray*}
	where we choose for a fixed number $\mu \in (0,1)$ with
	\[
	M := \max\Big\{ \Big[L_f + \|C\|(1 +C_\alpha)\Big](1+\|C\|), \frac{1}{2}\Big\}
	\]
	and $T = T(a)$ satisfying
	\begin{equation*}
	(T-a)^{1-2\alpha} + \ltn x \rtn_{\alpha, [a,T]} + \ltn \X \rtn_{2\alpha, \Delta^2([a,T])}^{\frac{1}{2}} = \frac{\mu}{2M}<1. 
	\end{equation*}
	Therefore, if we restrict to the set
	\[
	\cB := \Big\{ (y,y^\prime) \in \cD^{2\alpha}_x(y_a,C y_a), \ltn y,y^\prime\rtn_{x,2\alpha} \leq \frac{\mu}{1- \mu} \|y_a\| \Big\}
	\]
	then
	\begin{eqnarray*}
		\ltn \cM(y,y^\prime) \rtn_{x, 2\alpha} &\leq & \mu \|y,y^\prime\|_{x,2\alpha}\leq\Big(\frac{\mu^2}{1- \mu}+\mu\Big)\|y_a\| \leq \frac{\mu}{1-\mu}\|y_a\|,
	\end{eqnarray*}
	which proves that $\cM: \cB \to \cB$. By Schauder-Tichonorff theorem, there exists a  fixed point of $\cM$ which is a solution of equation \eqref{RDE1} on the interval $[a,T]$. Next, for any two solutions $(y,y^\prime), (\bar{y},\bar{y}^\prime)$ of the same initial conditions $(y_a, Cy_a)$, by similar computations, one get
	\begin{eqnarray*}
	\ltn (y,y^\prime) - (\bar{y},\bar{y}^\prime)\rtn_{x,2\alpha}
	&\leq& \mu \Big(\|y_a-\bar{y}_a\| + \ltn (y,y^\prime) - (\bar{y},\bar{y}^\prime) \rtn_{x,2\alpha}\Big) \leq  \mu \ltn (y,y^\prime)-(\bar{y},\bar{y}^\prime)\rtn_{x,2\alpha} 
	\end{eqnarray*}
	and together with $\mu <1$, this proves the uniqueness of solution of \eqref{RDE1} on $[a,T]$. By constructing the greedy time sequence \eqref{greedywitht}, we can extend and prove the existence of the unique solution on the whole real line. It is easy to see that solution $y_t$ depends linearly on initial $y_a$, hence there exists a solution matrix $\Phi(t,a,x,\X)$ of equation \eqref{linRDE3}. The similar conclusion holds for the backward equation.	
\end{proof}	
The estimate of the solution under the supremum norm $\|\cdot\|_\infty$ and the $\ltn \cdot,\cdot \rtn_{x,2\alpha}$ semi-norm is proved straight forward. 
\begin{theorem}
	Assume $G(y)= Cy$. For any interval $[a,b]$, the seminorm $\ltn y,y^\prime \rtn_{x,2\alpha,[a,b]}$ and the supremum norm $\|y\|_{\infty,[a,b]}$ are estimated as follows.
	\begin{eqnarray}
	\ltn y,y^\prime \rtn_{x,2\alpha,[a,b]} 
	&\leq& \|y_a\|\exp \Big\{\bar{N}_{\frac{\mu}{M},[a,b],\alpha}(\bx)\log\Big(\mu+\frac{1}{1-\mu}\Big) \Big\}; \\
	\|y\|_{\infty,[a,b]} &\leq& \|y_a\|\exp \Big\{\bar{N}_{\frac{\mu}{M},[a,b],\alpha}(\bx)\log\Big(\mu+\frac{1}{1-\mu}\Big) \Big\}. \label{ysupest}
	\end{eqnarray}
	
\end{theorem}
\begin{proof}
	To estimate $\ltn y,y^\prime \rtn_{x,2\alpha}$, we use the same greedy time \eqref{greedywitht} to get
	\begin{eqnarray*}
		\ltn y,y^\prime \rtn_{x,2\alpha,[\bar{\tau}_i,\bar{\tau}_{i+1}]} &\leq&  \frac{\mu}{1-\mu} \|y_{\bar{\tau}_i}\| 
	\end{eqnarray*}
	so that
	\begin{eqnarray}\label{ytaui}
	\|y_{\bar{\tau}_{i+1}}\| &\leq& \|y\|_{\infty,[\tau_i,\tau_{i+1}]} \leq \|y_{\bar{\tau}_i}\| + \|C\| \|y_{\bar{\tau}_i}\| (\bar{\tau}_{i+1}-\bar{\tau}_i)^\alpha \ltn x \rtn_\alpha + \ltn y,y^\prime \rtn_{x,2\alpha,[\bar{\tau}_i,\bar{\tau}_{i+1}]} \notag\\
	&\leq&  \Big(1+ \|C\| (\bar{\tau}_{i+1}-\bar{\tau}_i)^\alpha \ltn x \rtn_\alpha + \frac{\mu}{1-\mu} \Big) \|y_{\bar{\tau}_i}\| \leq  \Big(\mu+\frac{1}{1-\mu}\Big)\|y_{\bar{\tau}_i}\|.
	\end{eqnarray}
	As a result
	\begin{equation*}
	\ltn y,y^\prime \rtn_{x,2\alpha,[\bar{\tau}_i,\bar{\tau}_{i+1}]} \leq \frac{\mu}{1-\mu} \|y_{\bar{\tau}_i}\| \leq \frac{\mu}{1-\mu}\Big(\mu+ \frac{1}{1-\mu}\Big)^i \|y_a\|
	\end{equation*}
	and therefore
	\begin{eqnarray*}
	\ltn y,y^\prime \rtn_{x,2\alpha,[a,b]} &\leq&\sum_{i=0}^{\bar{N}_{\frac{\mu}{M},[a,b],\alpha}(\bx)}	\ltn y,y^\prime \rtn_{x,2\alpha,[\bar{\tau}_i,\bar{\tau}_{i+1}]}\leq \sum_{i=0}^{\bar{N}_{\frac{\mu}{M},[a,b],\alpha}(\bx)}\frac{\mu}{1-\mu}\Big(\mu+\frac{1}{1-\mu}\Big)^i \|y_a\|\notag\\
	&\leq& \|y_a\|\exp \Big\{\bar{N}_{\frac{\mu}{M},[a,b],\alpha}(\bx)\log\Big(\mu+\frac{1}{1-\mu}\Big) \Big\}. 
	\end{eqnarray*}
	The same estimate using \eqref{ytaui} shows \eqref{ysupest}.
\end{proof}	

\section{Stability results}
We first present the definition of pathwise stability (see e.g. \cite{duchongcong18}).
\begin{definition}\label{Defstability}
	(A) Stability: A solution $\mu(\cdot)$ of the deterministic differential equation \eqref{RDE1} is called stable, if for any $\varepsilon >0$ there exists an $r =r(\varepsilon)>0$ such that for any solution $y$ of  \eqref{RDE1} satisfying  $\|y_a-\mu_a\| < r$ the following inequality holds 
	\[
	\sup_{t\geq a}\|y_t-\mu_t\| < \varepsilon.
	\]
	(B) Attractivity: $\mu$ is called attractive, if  there exists $r >0$ such that for any solution $y$ of  \eqref{RDE1} satisfying  $\|y_a-\mu_a\| < r$ we have
	\[
	\lim \limits_{t \to \infty} \|y_t-\mu_t\| = 0.   
	\]	
	(C) Asymptotic stability: $\mu$ is called
	\begin{enum}
		\item asymptotically stable, if it is stable and attractive.
		\item exponentially  stable, if it is stable and there exists $r>0$ such that for any solution $y$ of  \eqref{RDE1} satisfying  $\|y_a-\mu_a\| < r$ we have
		\[
		\limsup \limits_{t \to \infty} \frac{1}{t} \log\|y_t-\mu_t\|
		< 0. 
		\]
	\end{enum}
\end{definition}

\subsection{Case 1. $\nu \in (\frac{1}{2},1)$: Young systems}

\begin{lemma}\label{roughy}
	Let $\gamma(s,t)$ be a control function, $\Lambda([s,t])$ a positive increasing function w.r.t. the inclusion of interval set $[s,t]$. Assume $\theta \in C^{q-{\rm var}}$ satisfying for any $s,t \in [a,b]$
	\begin{equation}\label{condy}
	\ltn \theta \rtn_{q{\rm - var},[s,t]} 
	\leq \gamma(s,t) + \Lambda([s,t])\ltn x \rtn_{p-{\rm var},[s,t]} + 2K \Lambda([s,t])\ltn x \rtn_{p-{\rm var},[s,t]}\ltn \theta \rtn_{q-{\rm var},[s,t]}.
	\end{equation}
	Then for any $s,t \in [a,b]$
	\begin{equation}\label{concludey}
	\ltn \theta \rtn_{q-{\rm var},[s,t]} 
	\leq 2 \gamma(s,t) + 2 \Lambda([s,t])\ltn x \rtn_{p-{\rm var},[s,t]} + (2K)^{p-1} (2\Lambda([s,t]))^p \ltn x \rtn_{p-{\rm var},[s,t]}^p. 
	\end{equation}
\end{lemma}

\begin{proof}
	We apply the same arguments as in \cite[Proposition 5.10, pp.~83-84]{friz}. Namely, for any fixed $[s,t] \subset [a,b]$, it follows from \eqref{condy} that for $[u,v] \subset [s,t]$
	\begin{eqnarray}\label{yqvar}
	\ltn \theta \rtn_{q-{var},[u,v]} \leq 2 \gamma(u,v) + 2 \Lambda([s,t])\ltn x \rtn_{p-{\rm var},[u,v]} \quad\text{whenever} \quad \ltn x \rtn_{p-{\rm var},[u,v]} \leq \frac{1}{4 K\Lambda([s,t])}. 
	\end{eqnarray}
	Assume that $\ltn x \rtn_{p-{\rm var},[s,t]} > \frac{1}{4 K\Lambda([s,t])}$, define a sequence of greedy time 
	\[
	t_0 = s,\quad t_{i+1} := \inf \{ u \geq t_i, \ltn x \rtn_{p-{\rm var},[t_i,u]} = \frac{1}{4 K\Lambda([s,t])} \} \wedge t.
	\]
	The sequence would end up at some time $t_N = t$, with 
	\[
	(N-1)\Big(\frac{1}{4 K\Lambda([s,t])}\Big)^p= \sum_{i =0}^{N-1} \ltn x \rtn_{p-{\rm var},[t_i,t_{i+1}]}^p \leq \ltn x \rtn_{p-{\rm var},[s,t]}^p,  
	\] 
	so that
	\[
	N-1 \leq \Big(4K\Lambda([s,t])\Big)^p\ltn x \rtn_{p-{\rm var},[s,t]}^p.
	\]
	Together with \eqref{yqvar} and the greedy times $t_i$, we derive
		\allowdisplaybreaks
	\begin{eqnarray*}
		&&\|\theta_t-\theta_s\| \leq \sum_{i=0}^{N-1} \|\theta_{t_{i+1}}-\theta_{t_i}\| \\
		&\leq&  \sum_{i=0}^{N-2} \Big(2 \gamma(t_i,t_{i+1}) + 2 \Lambda([s,t]) \frac{1}{4 K\Lambda([s,t])}\Big) + 2 \gamma(t_{N-1},t_N) + 2\Lambda([s,t]) \ltn x \rtn_{p{\rm -var},[t_{N-1},t_N]} \\
		&\leq&  2 \gamma(s,t) +(N-1) \frac{1}{2K} + 2 \Lambda([s,t]) \ltn x \rtn_{p{\rm -var},[s,t]}\\
		&\leq& 2 \gamma(s,t) + \frac{1}{2K} \Big(4K\Lambda([s,t])\Big)^p \ltn x \rtn_{p-{\rm var},[s,t]}^p + 2 \Lambda([s,t]) \ltn x \rtn_{p{\rm -var},[s,t]},
	\end{eqnarray*}
	in case $\ltn x \rtn_{p-{\rm var},[s,t]} > \frac{1}{4 K\Lambda([s,t])}$. All in all, for any $s,t \in [a,b]$
	\[
	\|\theta_t-\theta_s\| \leq 2 \gamma(s,t) + 2 \Lambda([s,t])\ltn x \rtn_{p-{\rm var},[s,t]}+ (2K)^{p-1} \Big(2\Lambda([s,t])\Big)^p \ltn x \rtn_{p-{\rm var},[s,t]}^p.
	\]
	Using the fact that $\gamma(s,t)$ and $\ltn x \rtn_{p-{\rm var},[s,t]}^p$ are control functions, it follows from the definition of $q$-var seminorm that for all $a\leq s \leq t \leq b$
	\begin{equation*}
	\ltn \theta \rtn_{q-{\rm var},[s,t]} 
	\leq 2 \gamma(s,t) + 2 \Lambda([s,t])\ltn x \rtn_{p-{\rm var},[s,t]} 
	+ (2K)^{p-1} (2 \Lambda([s,t]))^p \ltn x \rtn_{p-{\rm var},[s,t]}^p.
	\end{equation*}
\end{proof}	
	\allowdisplaybreaks
\begin{lemma}\label{localstabPro}
	Assume that there exist positive increasing functions $H,\kappa_1,\kappa_2$ with 
	\begin{eqnarray}
	E \kappa_1 (\ltn x \rtn_{p-{\rm var},[0,1]}) &<& \infty; \label{kappaexpect}
	\end{eqnarray}
	such that $y_t$ satisfying
	\begin{equation}\label{logx20}
	\log \|y_t\| \leq \log \|y_a\| + \int_a^t [H(\|y_s\|)-\lambda_A]ds + C_g \kappa_1(\ltn x \rtn_{p-{\rm var},[a,t]}) +C_g \kappa_2(\|y_a\|),\quad \forall a\leq t \leq a+1. 
	\end{equation}
	If $H(0) < \lambda_A$ then there exists $\epsilon >0$ such that for all $C_g<\epsilon$ the zero solution is locally exponentially stable a.s. 
\end{lemma}

\begin{proof}
	We apply the random norm techniques in \cite[Chapter 6]{arnold} to translate the original problem for random integral inequality \eqref{logx20} into the problem for deterministic integral inequality. Fix an $0<\epsilon <\lambda_A - H(0)- \epsilon E \kappa_1 (\ltn x \rtn_{p-{\rm var},[0,1]})$ and assign
	\[
	\Gamma(t,x) := C_g\kappa_1(\ltn x \rtn_{p-{\rm var},[n,t]}) + \sum_{k=0}^{n-1}  C_g\kappa_1(\ltn x \rtn_{p-{\rm var},[k,k+1]}), \quad \forall n \geq 0, \forall t \in [n,n+1].
	\]
	Then it follows from \eqref{logx20} that
	\[
	\log \|y_t\| \leq \log \|y_n\| + \int_n^t [H(\|y_s\|)-\lambda_A]ds + C_g\kappa_1(\ltn x \rtn_{p-{\rm var},[n,t]}) + C_g \kappa_2(\|y_n\|),\quad \forall t \in [n,n+1]. 
	\]
	Hence for any $t \in [n,n+1]$
	\begin{eqnarray*}
		&&\log \|y_t\| \exp \{(\lambda_A - H(0) - \epsilon) t - \Gamma(t,x)\} \\
		&\leq& \log \|y_n\|  \exp\{(\lambda_A - H(0) - \epsilon) n -  \Gamma(n,x)\} \\
		&&+ \int_n^t \Big[H\Big(\|y_s\|\exp \{(\lambda_A - H(0) - \epsilon) s - \Gamma(s,x)\} \exp \{-(\lambda_A - H(0) - \epsilon) s + \Gamma(s,x)\}\Big)\\
		&&\qquad \qquad \qquad -(H(0) + \epsilon)\Big]ds\\
		&&+  C_g\kappa_2\Big(\|y_n\|\exp \{(\lambda_A - H(0) - \epsilon) n - \Gamma(n,x)\} \exp \{-(\lambda_A - H(0) - \epsilon) n + \Gamma(n,x)\}\Big).
	\end{eqnarray*}
	From the definitions of $\Gamma$ and $\kappa_1$, for almost sure all $x$ there exist the limit
	\begin{equation}\label{gamma}
	\lim \limits_{t \to \infty} \frac{\Gamma(t,x)}{t} = C_g\lim \limits_{n \to \infty} \frac{1}{n} \sum_{k=0}^{n-1}  \kappa_1(\ltn x \rtn_{p-{\rm var},[k,k+1]}) = C_g E \kappa_1 (\ltn x \rtn_{p-{\rm var},[0,1]})< \lambda_A - H(0) - \epsilon,
	\end{equation}
	thus there exists an integer $m=m(\lambda_A - H(0) - \epsilon,x)$ such that
	$-(\lambda_A - H(0) - \epsilon) t + \Gamma(t,x) < 0$ for any $t \geq m(\lambda_A - H(0) - \epsilon,x)$. Assign 
	\[
	z_t := \log \|y_t\| \exp\{(\lambda_A - H(0) - \epsilon) t - \Gamma(t,x)\} = \log \|y_t\| + (\lambda_A - H(0) - \epsilon) t - \Gamma(t,x), \forall t\geq0.
	\]
	Because $H$ and $\kappa_2$ are increasing functions, it follows that for any $n\geq m((\lambda_A - H(0) - \epsilon),x)$
	\begin{equation}\label{zeq}
	z_t \leq z_n + C_g \kappa_2 (e^{z_n})+ \int_n^t \Big[H(e^{z_s}) - (H(0) + \epsilon)\Big]ds, \quad \forall t \in [n,n+1].  
	\end{equation}
	Again since $H$ and $\kappa_2$ are increasing functions, there exists a $\delta >0$ such that 
	\[
	C_g \kappa_2(\delta) + H(\delta e^{C_g \kappa_2(\delta)}) < H(0) + \epsilon.
	\]
	Using \eqref{xqvar}, one can choose $r(x)$ such that
	\begin{equation}\label{x01}
	\|y_0\| < r(x) = \delta \exp \{\Gamma(m,x) - (\lambda_A - H(0) - \epsilon) m\} \prod_{j=0}^{m-1} \Big[ 1 + \exp \{F(\ltn x \rtn_{p-{\rm var},[j,j+1]})\}\Big]^{-1},
	\end{equation}
	so that \eqref{x01} and \eqref{xqvar} implies
	\[
	z_m = \log \|y_{m}\| + (\lambda_A - H(0) - \epsilon) m - \Gamma(m,x) < \log \delta,\quad \forall \|y_0\| < r(x).
	\]
	Because $H(\exp \{ z_m + C_g \kappa_2 (e^{z_m})\}) < H(\delta e^{C_g \kappa_2(\delta)})<H(0) + \epsilon$, it follows from the continuity in $s$ of $H(e^{z_s})$ that $H(e^{z_s}) <H(0) + \epsilon, \forall s \in [m, m+\tau)$ for some small $\tau>0$. Denote by $\tau_\infty$ the supremum of such $\tau$ and assume $\tau_\infty<1$, then the integral $\int_m^{m+\tau_\infty}[\dots]ds$ in \eqref{zeq} is negative, hence  $z_{m+\tau_\infty} <z_m + C_g \kappa_2 (e^{z_m}) < \log \delta + C_g \kappa_2(\delta)$ and $H(e^{z_{m+\tau_\infty}}) < H(\delta e^{C_g \kappa_2(\delta)}) < H(0) + \epsilon$. This means there exists $\tau_0 >\tau_\infty$ such that $H(e^{z_s}) <H(0) + \epsilon, \forall s \in [m, m+\tau_0)$ which contradicts to the definition of $\tau_\infty$. Therefore $\tau_\infty \geq 1$ and $z_t < \log  \delta + C_g \kappa_2(\delta), \forall t \in [m,m+1]$. Again \eqref{zeq} yields 
	\[
	z_t \leq z_m + C_g \kappa_2(\delta)- \Big[H(0) + \epsilon - H\Big(\delta e^{C_g \kappa_2(\delta)}\Big) \Big] (t-m), \forall t\in [m,m+1]
	\]
	and in particular
	\begin{equation}\label{zeq1}
	z_{m+1} \leq z_m - \Big[H(0) + \epsilon -H(\delta e^{C_g \kappa_2(\delta)}) - C_g\kappa_2(\delta) \Big] <  z_m < \log \delta.
	\end{equation}
	By the induction principle, \eqref{zeq1} holds for every $n\geq m$. Then for all $ t\in [n,n+1] $ with  $n \geq m$, we use \eqref{zeq1} to get
	\begin{eqnarray*}
		z_t &\leq& z_n +C_g \kappa_2(\delta)- \Big[H(0) + \epsilon- H\Big(\delta e^{C_g \kappa_2(\delta)}\Big) \Big] (t-n) \\
		&\leq& z_m- \Big[H(0) + \epsilon -H(\delta e^{C_g \kappa_2(\delta)}) - C_g \kappa_2(\delta) \Big](n-m)+ C_g \kappa_2(\delta)\\
		&&- \Big[H(0) + \epsilon - H\Big(\delta e^{ C_g \kappa_2(\delta)}\Big) \Big] (t-n)\\
		&\leq& \log \delta +  C_g \kappa_2(\delta) - \Big[H(0) + \epsilon -H(\delta e^{C_g \kappa_2(\delta)}) - C_g \kappa_2(\delta) \Big](t-m).
	\end{eqnarray*}
	 As a result,
	\begin{eqnarray*}
		&& \log \|y_t\| \\
		&\leq& \Gamma(t,x) - (\lambda_A - H(0) - \epsilon) t + \log \delta +  C_g\kappa_2(\delta) - \Big[H(0) + \epsilon -H(\delta e^{C_g \kappa_2(\delta)}) - C_g \kappa_2(\delta) \Big](t-m) \\
		&\leq&\Gamma(t,x) + \log \delta + C_g \kappa_2(\delta) + \Big[H(0) + \epsilon -H(\delta e^{C_g \kappa_2(\delta)}) - C_g \kappa_2(\delta) \Big]m- \Big[\lambda_A -H(\delta e^{C_g \kappa_2(\delta)}) - C_g \kappa_2(\delta) \Big]t
	\end{eqnarray*}
	thus
	\begin{eqnarray}
	\limsup \limits_{t\to \infty} \frac{1}{t} \log \|y_t\| &\leq& -  \Big[\lambda_A -H(\delta e^{ C_g \kappa_2(\delta)}) - C_g \kappa_2(\delta) \Big] +  E \kappa_1 (\ltn x \rtn_{p-{\rm var},[0,1]})\notag\\
	&\leq& -  \Big[H(0) + \epsilon - H(\delta e^{ C_g \kappa_2(\delta)}) -C_g \kappa_2(\delta) \Big] < 0.
	\end{eqnarray}
	In other words, by choosing $y_0$ satisfying \eqref{x01}, the zero solution is locally exponentially stable.  
	\end{proof}

\begin{lemma}\label{globalstabPro}
	Assume that there exist positive increasing functions $H,\kappa_1$ with 
	\begin{eqnarray}
	E \kappa_1 (\ltn x \rtn_{p-{\rm var},[0,1]}) &<& \infty; \label{kappaexpect}
	\end{eqnarray}
	such that $y_t$ satisfying
	\begin{equation}\label{logx20}
	\log \|y_t\| \leq \log \|y_a\| + \int_a^t [H(\|y_s\|)-\lambda_A]ds + C_g \kappa_1(\ltn x \rtn_{p-{\rm var},[a,t]}),\quad \forall a\leq t \leq a+1. 
	\end{equation}
	If $\|H\|_\infty < \lambda_A$ then there exists an $\epsilon >0$ such that for $C_g <\epsilon$, the zero solution is globally exponentially stable a.s.
\end{lemma}
	\begin{proof}
	We can choose $\epsilon$ such that given $C_g < \epsilon$ 
	\begin{equation}\label{eps2}
	0 < \lambda := \lambda_A - C_f - C_g E \kappa_1\Big(\ltn x \rtn_{p-{\rm var},[0,1]}\Big). 
	\end{equation}
	It follows from \eqref{logx20} that
	\[
	\log \|y_1\| \leq \log \|y_0\| -(\lambda_A -C_f) + C_g \kappa_1(\ltn x \rtn_{p-{\rm var},[0,1]}) 
	\] 
	or by induction for any $n \in \N$
	\begin{equation}\label{xn}
	\log \|y_n\| \leq \log \|y_0\| -\Big[\lambda_A -C_f- \frac{1}{n}\sum_{k=0}^{n-1} C_g \kappa_1(\ltn x \rtn_{p-{\rm var},[k,k+1]}) \Big]n.
	\end{equation}
	Using the ergodic Birkhorff theorem and \eqref{eps2}, we then get for a.s. all realization 
	\[
	\limsup \limits_{n \to \infty} \log \|y_n\| \leq \lambda_A -C_f-C-g E \kappa_1\Big(\ltn x \rtn_{p-{\rm var},[0,1]}\Big) = -\lambda <0,
	\]
	which proves the globally exponential stability of the zero solution.
\end{proof}

\begin{theorem}[Local stability for Young differential equations]\label{globalYDE}
	Assume $X_\cdot(\omega)$ is a Gaussian process satisfying \eqref{Gaussianexpect}, and $\bar{\nu}>\nu >\frac{1}{2}$ is fixed. Assume further that conditions \eqref{lambda}, \eqref{condf} are satisfied, where $\lambda_A > h(0)$. Then the zero solution of \eqref{RDE1} is locally exponentially stable for almost sure all the trajectories $x$ of $X$. If in addition $\lambda_A > C_f$, then we can choose $\epsilon$ so that the zero solution of \eqref{RDE1} is globally exponentially stable a.s. 
\end{theorem}

\begin{proof}
	We summarize the ideas of the proof here for reader benefits. In {\bf Step 1} we use the integration by parts to derive the equation of $\log \|y_t\|$ in \eqref{logxg} and the equation of $\theta_t = \frac{y_t}{\|y_t\|}$ in \eqref{y}. The estimate of  $\ltn \theta \rtn_{q-{\rm var},[s,t]}$ is then given by \eqref{yest} by applying Lemma \ref{roughy}. In {\bf Step 2} we derive an estimate of $\log \|y_t\|$ in \eqref{logxh}, with the help of auxilliary polinomials $P_i, i = 1,\dots,4$ satisfying \eqref{kappa}. The conclusion of local stability is then a direct consequence of Lemma \ref{localstabPro}. In case $\lambda_A > C_f$ we prove in {\bf Step 3} that $e^{2(\lambda_A-C_f)t} \|y_t\|^2$ satisfies \eqref{normy21} and \eqref{normy213}, hence the global exponential stability is followed by applying the discrete Gronwall lemma \cite[Lemma 4]{ducGANSch18} and choosing $C_g$ according to \eqref{normy214}. \\
	
{\bf Step 1.} As proved in \cite{congduchong17}, there exists a unique solution of \eqref{RDE2} and also the backward equation. Since $y\equiv 0$ is the solution of \eqref{RDE2}, it follows that $y_t \ne 0$ for all $t\in \R$ if $y_0 \ne 0$ (otherwise there would be two solutions of the backward equation starting from $y_t$ and ending at zero and $y_0$, which is a contradiction). Then observe that
\begin{equation}\label{gx}
\frac{g(y_s)}{\|y_s\|} = \frac{g(y_s)-g(0)}{\|y_s\|} = \frac{\int_0^1 D_yg(\eta y_s) y_s d\eta}{\|y_s\|} = \int_0^1 D_y g(\eta y_s) \theta_s d\eta =: G(y_s,\theta_s),\quad \forall s\in \R;
\end{equation}
meanwhile 
\[
\|f(y_s)\| = \|f(y_s) - f(0)\| \leq h(\|y_s\|)\|y_s\|, \quad \forall s \in \R. 
\]
Using the rule of integration by parts (see \cite{zahle, zahle2}), it is easy to check that
\begin{equation}\label{logxg}
d \log \|y_t\| = \langle \theta_t, A \theta_t + \frac{f(y_t)}{\|y_t\|} \rangle dt + \langle \theta_t, G(y_t,\theta_t) \rangle dx_t,
\end{equation}
where $\theta_t$ satisfies the equation
\begin{equation}\label{y}
d \theta_t = \Big(A \theta_t + \frac{f(y_t)}{\|y_t\|} - \theta_t \langle \theta_t, A \theta_t + \frac{f(y_t)}{\|y_t\|} \rangle \Big)dt + \Big(G(y_t,\theta_t) - \theta_t \langle \theta_t, G(y_t,\theta_t) \rangle\Big)dx_t.
\end{equation} 
A direct computation using assumptions shows that $\|G(y,\theta)\|_{\infty,[a,b]} \leq C_g $ and 
\allowdisplaybreaks
\begin{eqnarray}\label{G}
\ltn G(y,\theta) \rtn_{q-{\rm var},[a,b]} &=& \ltn \int_0^1 D_y g(\eta y) \theta d\eta \rtn_{q-{\rm var},[a,b]} \notag\\
&\leq& \ltn \int_0^1 D_y g(\eta y) d\eta \rtn_{q-{\rm var},[a,b]} \|\theta\|_{\infty,a,b} + \left\| \int_0^1 D_y g(\eta y) d\eta \right \|_{\infty,[a,b]} \ltn \theta \rtn_{q{\rm - var},[a,b]} \notag \\
&\leq& C_g \Big(\ltn \theta \rtn_{q-{\rm var},[a,b]} + \frac{1}{2}\ltn y \rtn_{q-{\rm var},[a,b]}\Big).
\end{eqnarray}
it follows that
\begin{eqnarray*}
\|\theta_t - \theta_s\| &\leq& 2 \|A\|(t-s) + 2 \int_s^t h(\|y_u\|)du + 2C_g\ltn x \rtn_{p{\rm -var},[s,t]} \\
&& + K \ltn x \rtn_{p{\rm -var},[a,t]} \ltn G(y,\theta)- \theta \langle \theta, G(y,\theta)\rangle \rtn_{q{\rm - var},[s,t]} \\
&\leq& 2 \|A\|(t-s) + 2 \int_s^t h(\|y_u\|)du +2C_g\ltn x \rtn_{p-{\rm var},[s,t]}  \\
&&+ K C_g\ltn x \rtn_{p-{\rm var},[s,t]} \ltn y \rtn_{q-{\rm var},[s,t]}+ 4 KC_g \ltn x \rtn_{p-{\rm var},[s,t]}\ltn \theta \rtn_{q-{\rm var},[s,t]}.
\end{eqnarray*}
Since each of $t-s, \int_s^t h(\|y_u\|)du, \ltn x \rtn_{p-{\rm var},[s,t]} \ltn x \rtn_{q-{\rm var},[s,t]}$ is a control, the function
\[
\gamma(s,t):= 2 \|A\|(t-s) + 2 \int_s^t h(\|y_u\|)du + K C_g\ltn x \rtn_{p-{\rm var},[s,t]} \ltn y \rtn_{q-{\rm var},[s,t]}
\]
is also a control. By using triangle inequality for $q$-var seminorm with $q\geq p \geq 1$, we get for all $a \leq s < t \leq b$
\begin{eqnarray*}
&&\|\theta_t-\theta_s\| \leq \ltn \theta \rtn_{q{\rm - var},[s,t]} \notag\\
&=& \sup_{\Pi} \Big\{\sum_{[u,v] \in \Pi} \Big(\gamma(u,v) + 2 C_g\ltn x \rtn_{p-{\rm var},[u,v]}  + 4 KC_g\ltn x \rtn_{p-{\rm var},[s,t]}\ltn \theta \rtn_{q-{\rm var},[u,v]}\Big)^q \Big\}^{\frac{1}{q}} \notag \\
&\leq& \gamma(s,t) + 2C_g \ltn x \rtn_{p-{\rm var},[s,t]} + 4 KC_g\ltn x \rtn_{p-{\rm var},[s,t]}\ltn \theta \rtn_{q-{\rm var},[s,t]},
\end{eqnarray*}
which has the form of \eqref{condy} with $\Lambda([s,t]) := 2 C_g$. Applying \eqref{concludey} in Lemma \ref{roughy} we conclude that for all $a\leq s \leq t \leq b$
\begin{eqnarray}\label{yest}
\ltn \theta \rtn_{q-{\rm var},[s,t]} &\leq& 2 \gamma(s,t) + 4 C_g \ltn x \rtn_{p-{\rm var},[s,t]}+ (2K)^{p-1} (4 C_g)^p \ltn x \rtn_{p-{\rm var},[s,t]}^p \notag \\
&\leq&  4 \|A\|(t-s) + 4 \int_s^t h(\|y_u\|)du + 2 K C_g\ltn x \rtn_{p-{\rm var},[s,t]} \ltn y \rtn_{q-{\rm var},[s,t]}\notag\\
&&+ 4C_g\ltn x \rtn_{p-{\rm var},[s,t]} + (2K)^{p-1} (4C_g)^p \ltn x \rtn_{p-{\rm var},[s,t]}^p.
\end{eqnarray}

{\bf Step 2.} Next, to estimate \eqref{logxg}, we first use \eqref{YL0} and \eqref{G} to get
\begin{eqnarray*}
	\left \| \int_a^{b} \langle \theta_s, G(y_s,\theta_s) \rangle dx_s \right \|
	&\leq& 	C_g \ltn x \rtn_{p-{\rm var},[a,b]} + K\ltn x \rtn_{p-{\rm var},[a,b]} \ltn \langle \theta,G(y,\theta)\rangle \rtn_{q{\rm -var},[a,b]}\\
	&\leq&\ltn x \rtn_{p-{\rm var},[a,b]} \Big(C_g + 2KC_g \ltn \theta \rtn_{q-{var},[a,b]} + \frac{1}{2}K C_g \ltn y \rtn_{q{\rm -var},[a,b]} \Big). 
\end{eqnarray*} 
We estimate equation \eqref{logxg} in the integration form, using \eqref{yest} and \eqref{lambda}
\allowdisplaybreaks
\begin{eqnarray*}
\log \|y_t\|&\leq& \log \|y_a\| + \int_a^t [-\lambda_A + h(\|y_s\|)]ds \notag\\
&&+ \ltn x \rtn_{p{\rm -var},[a,t]} \Big( C_g + 2KC_g \ltn \theta \rtn_{q-{var},[a,b]} + \frac{1}{2}K C_g \ltn y \rtn_{q{\rm -var},[a,b]}\Big) \notag\\
&\leq& \log\|y_a\| + \int_a^t [-\lambda_A + h(\|y_s\|)]ds +  C_g\ltn x \rtn_{p{\rm -var},[a,t]} + \frac{1}{2} KC_g \ltn x \rtn_{p{\rm -var},[a,t]} \ltn y \rtn_{q{\rm - var},[a,t]} \notag \\
&& + 2KC_g \ltn x \rtn_{p{\rm -var},[a,t]} \Big\{ 4 (\|A\|+C_f)(t-a) + 2 K C_g\ltn x \rtn_{p-{\rm var},[a,t]} \ltn y \rtn_{q-{\rm var},[a,t]} \notag\\
&&\qquad \qquad \qquad \qquad \qquad \qquad + 4C_g \ltn x \rtn_{p-{\rm var},[a,t]} + (2K)^{p-1} (4C_g)^p \ltn x \rtn_{p-{\rm var},[a,t]}^p\Big\}.
\end{eqnarray*}
Writing in short $Y_{a,t} := \ltn y\rtn_{q-{\rm var},[a,t]}$ and $x_p = \ltn x \rtn_{p-{\rm var},[a,t]}$, we then get for all $0\leq a < t \leq a+1$
\begin{eqnarray}\label{logxest}
\log \|y_t\|&\leq& \log\|y_a\| + \int_a^t [-\lambda_A + h(\|y_s\|)]ds + C_gx_p  + \frac{1}{2}K C_g x_p Y_{a,t} \notag\\
&&+ 2K C_g x_p \Big\{ 4 (\|A\| + C_f)(t-a) + 2K C_g x_p Y_{a,t}+ 4C_g x_p + (2K)^{p-1} (4C_gx_p)^p\Big\} \notag\\
&\leq& \log\|y_a\| + \int_a^t [-\lambda_A + h(\|y_s\|)]ds + \Big( \frac{1}{2}KC_g x_p + 4K^2C_g^2 x_p^2\Big)Y_{a,t} \notag\\
&&+ C_g\Big[x_p + 8K(\|A\|+C_f) x_p + 8KC_g x_p^2 + (8KC_g)^p x_p^{p+1} \Big].
\end{eqnarray}
On the other hand, it follows from \eqref{xqvar} and Cauchy inequality that  
\begin{eqnarray*}
\Big( \frac{1}{2}KC_g x_p + 4K^2 C_g^2 x_p^2\Big) Y_{a,t}&\leq& \Big( \frac{1}{2}K  x_p + 4K^2 C_g x_p^2\Big) C_g\|y_a\| \exp \Big\{ F(x_p) \Big\}\\
&\leq& \frac{1}{2}C_g \|y_a\|^2 + \frac{1}{2} C_g\Big( \frac{1}{2}K x_p + 4K^2 C_g x_p^2\Big)^2 \exp \Big\{ 2F(x_p) \Big\}.
\end{eqnarray*}
In summary, we have just proved that for all $a\leq t\leq a+1$
\begin{eqnarray}\label{logxh}
\log \|y_t\| 
&\leq& \log \|y_a\| + \int_a^t \Big[-\lambda_A + h(\|y_s\|) \Big]ds + C_g P(x_p, e^{x_p}) + \frac{1}{2}C_g \|y_a\|^2
\end{eqnarray}
where $P(x_1,x_2)$ is a polynomial with positive coefficients depending on $C_g$ such that
\begin{eqnarray}\label{kappa}
P (0,x_2) = P(x_1,0) = 0.
\end{eqnarray}
Assign
\begin{eqnarray}\label{coefs}
H(z) := h(z),\qquad \kappa_1(z) := P(z,e^z),\qquad \kappa_2(z) := \frac{1}{2}z^2. 
\end{eqnarray}
Since the random variable $Z:= e^{ \ltn x\rtn_{p-{\rm var},[0,1]}}$ has finite moments of any order for $1<p<2$ and $x$ to be a realization of Gaussian stochastic process, it follows that $\kappa_1$ satisfies \eqref{kappaexpect}. Hence using $\lambda_A > h(0)$ the conclusion of local stability is therefore a direct consequence of Lemma \ref{localstabPro}. \\

{\bf Step 3.} Assume $\lambda_A > C_f$ and assign $\lambda := \lambda_A - C_f >0$, then we apply the integration by parts to get
\begin{equation*}
de^{2 \lambda t}\|y_t\|^2 = 2\lambda e^{2 \lambda t} \|y_t\|^2 dt + 2 e^{2 \lambda t} \langle y_t, A y_t + f(y_t) \rangle dt + 2 e^{2 \lambda t} \langle y_t, g(y_t) \rangle dx_t,  
\end{equation*}
or in the integral form
\begin{equation}\label{normy21}
e^{2 \lambda t} \|y_t\|^2 = \|y_0\|^2 + 2 \int_0^t e^{2 \lambda s} \Big(\lambda \|y_s\|^2 + \langle y_s, Ay_s+ f(y_s) \rangle\Big)ds + 2 \int_0^t e^{2 \lambda s} \langle y_s, g(y_s) \rangle dx_s.
\end{equation}
Using \eqref{lambda}, the first integral in \eqref{normy21} is then non-positive, thus for any $n \in \N$
\begin{eqnarray}\label{normy212}
e^{2 \lambda n} \|y_n\|^2 &\leq& \|y_0\|^2 + \sum_{k=0}^{n-1} 2\Big\|\int_k^{k+1} e^{2 \lambda s} \langle y_s, g(y_s) \rangle dx_s\Big\| \notag \\
&\leq& \|y_0\|^2 +  \sum_{k=0}^{n-1} 2\ltn x \rtn_{p{\rm - var},[k,k+1]}\Big( e^{2\lambda k} \|\langle y_k, g(y_k) \rangle\|+ K \ltn e^{2\lambda \cdot} \langle y, g(y)\rangle \rtn_{q{\rm -var},[k,k+1]}\Big).\notag\\
\end{eqnarray}
Observe that $\|\langle y_s, g(y_s)\| \leq C_g \|y_s\|^2$ and due to \eqref{xqvar}
\begin{eqnarray*}
	\ltn e^{2\lambda \cdot} \langle y, g(y)\rangle \rtn_{q{\rm -var},[s,t]}  &\leq& \ltn e^{2\lambda \cdot} \rtn_{q{\rm -var},[s,t]} C_g\|y\|_{\infty,[s,t]}^2 + 2e^{2\lambda t} C_g \|y\|_{\infty,[s,t]} \ltn y \rtn_{q{\rm -var},[s,t]} \\
	&\leq& \Big(e^{2\lambda t} - e^{2\lambda s}\Big) C_g \Big[1+ \exp \{F(\ltn x \rtn_{q{\rm -var},[s,t]})\}\Big]^2 \|y_s\|^2  \\
	&&+ 2e^{2\lambda t} C_g \Big[1+ \exp \{F(\ltn x \rtn_{q{\rm -var},[s,t]})\}\Big] \exp \{F(\ltn x \rtn_{q{\rm -var},[s,t]})\} \|y_s\|^2 \\
	&\leq& C_g e^{2\lambda s} \|y_s\|^2 \Big\{ \Big(e^{2\lambda (t-s)} - 1\Big) \Big[1+ \exp \{F(\ltn x \rtn_{q{\rm -var},[s,t]})\}\Big]^2 \\
	&&+ 2e^{2\lambda (t-s)} \Big[1+ \exp \{F(\ltn x \rtn_{q{\rm -var},[s,t]})\}\Big] \exp \{F(\ltn x \rtn_{q{\rm -var},[s,t]})\} \Big\}\\
	&\leq& C_g e^{2\lambda s} \|y_s\|^2 \kappa(t-s,\ltn x \rtn_{q{\rm -var},[s,t]} ),
\end{eqnarray*}
where
\[
\kappa(u,v) := (e^{2\lambda u}-1) [1+ e^{F(v)}]^2 + 2 e^{2\lambda u} [1+e^{F(v)}]e^{F(v)}.
\]
Hence it follows from \eqref{normy212} that
\begin{eqnarray}\label{normy213}
e^{2 \lambda n} \|y_n\|^2 &\leq& \|y_0\|^2 +  \sum_{k=0}^{n-1} 2C_g\ltn x \rtn_{p{\rm - var},[k,k+1]} \Big[\kappa(1,\ltn x \rtn_{p{\rm -var},[k,k+1]})+1\Big] e^{2\lambda k} \|y_k\|^2. 
\end{eqnarray} 
Applying the discrete Gronwall lemma in \cite[Lemma 4]{ducGANSch18} for the sequence $e^{2 \lambda n} \|y_n\|^2$ with parameters $2C_g\ltn x \rtn_{p{\rm - var},[k,k+1]} \Big[\kappa(1,\ltn x \rtn_{p{\rm -var},[k,k+1]})+1\Big]$ in \eqref{normy213}, we get
\[
e^{2 \lambda n} \|y_n\|^2 \leq \|y_0\|^2 \prod_{k=0}^{n-1} \Big(1+2C_g\ltn x \rtn_{p{\rm - var},[k,k+1]} \Big[\kappa(1,\ltn x \rtn_{p{\rm -var},[k,k+1]})+1\Big]\Big).
\]
Taking the logarithm on both sides, then dividing by $2n$ and letting $n$ tend to infinity we get, due to the inequality $\log(1+r) \leq r, \forall r>0$ and the ergodic Birkhorff theorem, that
\begin{eqnarray*}
\limsup \limits_{n \to \infty} \frac{1}{n} \log \|y_n\| &\leq& -\lambda + \lim \limits_{n \to \infty} \frac{1}{2n} \sum_{k =0}^{n-1} \log \Big(1+2C_g\ltn x \rtn_{p{\rm - var},[k,k+1]} \Big[\kappa(1,\ltn x \rtn_{p{\rm -var},[k,k+1]})+1\Big]\Big) \notag\\
&\leq& -\lambda + \lim \limits_{n \to \infty} \frac{1}{n} \sum_{k =0}^{n-1} C_g\ltn x \rtn_{p{\rm - var},[k,k+1]} \Big[\kappa(1,\ltn x \rtn_{p{\rm -var},[k,k+1]})+1\Big] \notag\\
&\leq& -\lambda + C_g E \ltn x \rtn_{p{\rm - var},[0,1]} \Big[\kappa(1,\ltn x \rtn_{p{\rm -var},[0,1]})+1\Big],
\end{eqnarray*}
where the expectation $E \ltn x \rtn_{p{\rm - var},[0,1]} \Big[\kappa(1,\ltn x \rtn_{p{\rm -var},[0,1]})+1\Big]$ is finite due to the fact that $E e^{\Lambda F(\ltn x \rtn_{p{\rm -var},[0,1]})}$ is finite for any $\Lambda >0$. Finally, for any $t \in [n,n+1]$, we use \eqref{xqvar} to get
\begin{eqnarray}\label{Lyaexp}
\limsup \limits_{t \to \infty} \frac{1}{t} \log \|y_t\| &\leq& \limsup \limits_{n \to \infty} \frac{1}{n} \log \|y_n\| + \limsup \limits_{n \to \infty} \frac{1}{n} \log [1 + \exp\{F(\ltn x \rtn_{p{\rm -var},[n,n+1]})\}] \notag\\
&\leq& -\lambda + C_g E \ltn x \rtn_{p{\rm - var},[0,1]} \Big[\kappa(1,\ltn x \rtn_{p{\rm -var},[0,1]})+1\Big] \notag\\
&&+ \limsup \limits_{n \to \infty} \frac{1}{n} [1 + F(\ltn x \rtn_{p{\rm -var},[n,n+1]})] \notag\\
&\leq& -\lambda + C_g E \ltn x \rtn_{p{\rm - var},[0,1]} \Big[\kappa(1,\ltn x \rtn_{p{\rm -var},[0,1]})+1\Big],
\end{eqnarray}
where the second limsup in the right hand side of \eqref{Lyaexp} is zero due to the integrability of $F(\ltn x \rtn_{p{\rm -var},[0,1]})$. Hence if we choose 
\begin{equation}\label{normy214}
C_g < \epsilon < (\lambda_A - C_f)\frac{1}{E \ltn x \rtn_{p{\rm - var},[0,1]} \Big[\kappa(1,\ltn x \rtn_{p{\rm -var},[0,1]})+1\Big]},
\end{equation}
then the zero solution is globally exponentially stable a.s.
\end{proof}

\begin{corollary}
	Assume that the linear Young differential equation
	\begin{equation}\label{linYDE}
	dy_t = Ay_t dt + C y_t dx 
	\end{equation}
	satisfies \eqref{lambda}. Then a criterion for the globally exponential stability is
	\begin{equation}\label{stablin}
	\lambda_A >  \Big (1 + 4K\|A\| + 8K + (8K)^p\Big)\|C\| \Big(E\ltn x \rtn_{p{\rm -var},[0,1]}^{p+1}\Big)^{\frac{1}{p+1}}
	\end{equation}
	
\end{corollary}
\begin{proof}
For $h(\cdot) \equiv 0, g(y)  = Cy$, there is no term $Y_{a,t}$ in \eqref{logxest}, hence it follows from \eqref{logxh} that
\begin{eqnarray*}
\log \|y_{k+1}\|&\leq& \log \|y_k\| - \lambda_A + (1 + 4K\|A\|)\|C\| \ltn x \rtn_{p{\rm -var},[k,k+1]}+ 8K \|C\|^2 \ltn x \rtn_{p{\rm -var},[k,k+1]}^2 \\
&&+ (8K)^p\|C\|^{p+1} \ltn x \rtn_{p{\rm -var},[k,k+1]}^{p+1}.
\end{eqnarray*}	
As a result
\begin{eqnarray*}
\limsup \limits_{n \to \infty} \frac{1}{n}\log \|y_n\| &\leq& - \lambda_A + (1 + 4K\|A\|)\|C\| E \ltn x \rtn_{p{\rm -var},[0,1]}+ 8K \|C\|^2 E \ltn x \rtn_{p{\rm -var},[0,1]}^2 \notag\\
&&+ (8K)^p\|C\|^{p+1} E\ltn x \rtn_{p{\rm -var},[0,1]}^{p+1},\\
&\leq& - \lambda_A + (1 + 4K\|A\|)\|C\| \Big(E\ltn x \rtn_{p{\rm -var},[0,1]}^{p+1}\Big)^{\frac{1}{p+1}}+ 8K \|C\|^2 \Big(E\ltn x \rtn_{p{\rm -var},[0,1]}^{p+1}\Big)^{\frac{2}{p+1}} \notag\\
&&+ (8K)^p\|C\|^{p+1} E\ltn x \rtn_{p{\rm -var},[0,1]}^{p+1}.
\end{eqnarray*}
Assign $\tilde{C} := \|C\|\Big(E\ltn x \rtn_{p{\rm -var},[0,1]}^{p+1}\Big)^{\frac{1}{p+1}}$, then system \eqref{linYDE} is exponentially stable if
\begin{equation}\label{AC}
\lambda_A > (1 + 4K\|A\|)\tilde{C}+ 8K \tilde{C}^2 + (8K)^p\tilde{C}^{p+1},
\end{equation}
which, together with the fact that $\lambda_A < \|A\|$ and $K>1$, implies that $\tilde{C} <1$. In that case \eqref{AC} is followed from \eqref{stablin}.
\end{proof}	


\subsection{Case 2. $\nu \in (\frac{1}{3},\frac{1}{2})$ and $g(y) = Cy$}

In this section we consider a particular rough case in which $g(y) = Cy$. We could then prove the same conclusions on stability, and even a general form of local stability. 

\begin{theorem}[Local stability for rough differential equations]\label{stabYDE1}
	Assume $\frac{1}{2}>\bar{\nu}> \nu >\frac{1}{3}$ and $X_\cdot(\omega)$ is a stationary process satisfying \eqref{Gaussianexpect}. Assume further that conditions \eqref{lambda}, \eqref{condf} are satisfied, where $g(y) = C y$ and $\lambda > h(0)$. Then there exists an $\epsilon >0$ such that given $\|C\| < \epsilon$, the zero solution of \eqref{RDE1} is locally exponentially stable for almost all realization $x$ of $X$. If in addition $\lambda > C_f$, then we can choose $\epsilon$ so that the zero solution of \eqref{RDE1} is globally exponentially stable. 
\end{theorem}

\begin{proof}
	We sketch out the proof here in several steps. In {\bf Step 1}, we derive the equation for $\log \|y_t\|$ in \eqref{logx}, and the equation for $\theta = \frac{y}{\|y\|}$ in \eqref{RDEangle}. Notice that for Gaussian geometric rough path, then $[x]_{\cdot,\cdot} = 0$, but we still compute the estimates here for general rough paths. As such the estimate for $\ltn \theta, \theta^\prime \rtn_{x,2\alpha,[a,b]}$ is proved by Proposition \ref{roughGronwall} which, due to $G(y)= Cy$, does not include $\ltn y, y^\prime \rtn_{x,2\alpha,[a,b]}$, hence we do not need the integrability of $\ltn y, y^\prime \rtn_{x,2\alpha,[a,b]}$. The estimate for $\log \|y_t\|$ is then derived in \eqref{logx2} in {\bf Step 2}, where each component is computed so that finally $\log \|y_t\|$ satisfies \eqref{logx3}. The conclusion is then followed from Proposition \ref{roughGronwall} and Theorem \ref{globalYDE}.  \\
	
	{\bf Step 1.} We use similar arguments in \cite{duchongcong18} to prove that the solution of the pathwise solution of the linear rough differential equation \eqref{RDE1} generates a linear rough flow on $\R^d$, and that $y_t = 0$ iff $y_0=0$. Hence it remains to prove all the formula for $\theta_t$ and $r_t$. By direct computations using \eqref{Roughformula}, we can show the following equations. 
	\begin{itemize}
		\item $\|y_t\|^2$ satisfies the RDE
		\begin{equation*}
		d\|y_t\|^2 = 2 \langle y_t, Ay_t + f(y_t) \rangle dt + 2 \langle y_t, Cy_t\rangle dx_t + \|C y_t\|^2 d[x]_{0,t},  
		\end{equation*}
		where $2\langle y,Cy \rangle^\prime_s = 2 \langle y^\prime_s, Cy_s \rangle + 2 \langle y_s, [Cy]^\prime_s \rangle$.
		\item $\|y_t\|$ satisfies the RDE
		\begin{eqnarray*}
		d\|y_t\| &=& \frac{1}{\|y_t\|} \langle y_t, Ay_t + f(y_t)\rangle dt + \frac{1}{\|y_t\|} \langle y_t, Cy_t \rangle dx_t \\
		&& + \frac{1}{2\|y_t\|}\Big[\|Cy_t\|^2 -\frac{1}{\|y_t\|^2} \langle y_t, Cy_t \rangle^2\Big]d[x]_{0,t},  
		\end{eqnarray*}
		where $\Big[\frac{1}{\|y\|}\langle y,Cy \rangle\Big] ^\prime_s = \Big[ \frac{1}{\|y\|}\Big]^\prime_s \langle y_s,Cy_s \rangle + \frac{1}{\|y_s\|} \Big[\langle y,Cy\rangle\Big]^\prime_s$.
		\item $\log \|y_t\|$ satisfies the RDE
		\begin{eqnarray}\label{logx}
		d\log \|y_t\| &=& \langle \theta_t, A\theta_t + \frac{f(y_t)}{\|y_t\|}\rangle dt +  \langle \theta_t, C\theta_t\rangle dx_t + \Big[\frac{1}{2}\|C\theta_t\|^2 - \langle \theta_t, C\theta_t \rangle^2\Big]d[x]_{0,t}, 
		\end{eqnarray}
		where $\Big[\langle \theta, C\theta \rangle\Big]^\prime_s =  \langle \theta^\prime_s,C\theta_s \rangle + \langle \theta_s, [C\theta]^\prime_s \rangle$.
		\item $\theta_t$ satisfies the RDE
		\begin{eqnarray}\label{RDEangle}
		d\theta_t&=& \Big[A \theta_t -  \langle \theta_t, A\theta_t \rangle \theta_t + \frac{f(y_t)}{\|y_t\|} - \langle \theta_t,\frac{f(y_t)}{\|y_t\|} \rangle \theta_t \Big]dt + \Big[C\theta_t -  \langle \theta_t, C\theta_t \rangle \theta_t \Big] dx_t \notag\\
		&&+ \frac{1}{2}\Big\{ 3\langle \theta_t, C\theta_t\rangle^2 \theta_t- 2 \langle \theta_t, C\theta_t\rangle C\theta_t  - \|C\theta_t\|^2\theta_t  \Big\}d[x]_{0,t},
		\end{eqnarray}
		where 
		\[
		\Big[C\theta - \langle \theta,C\theta \rangle \theta\Big]^\prime_s = [C\theta]^\prime_s - \langle \theta_s,C\theta_s\rangle \theta^\prime_s - \Big[\langle \theta_s,C\theta_s\rangle \Big]^\prime_s \theta_s.
		\]
	\end{itemize}
Rewrite \eqref{RDEangle} in the form
\begin{equation}\label{nonlinRDE1}
d\theta_t = f_1(t,\theta_t)dt + g_1(\theta_t)dx_t + k_1(\theta_t)d[x]_{a,t},\qquad t \in [a,b]
\end{equation}
or in the integral form
\begin{equation*}\label{nonlinRDE2}
\theta_t = F(\theta,\theta^\prime)_t = \theta_a + \int_a^t f_1(u,\theta_u)du + \int_a^t g_1(\theta_u)dx_u + \int_a^t k_1(\theta_u)d[x]_{a,u},\qquad \forall 0 \leq a \leq t \leq b;
\end{equation*}
where $g \in C^2$ such that there exist 
\[
C_{g_1} := \max \Big\{\|g_1(\theta)\|_{\infty,[0,T]}, \|D_\theta g_1(\theta)\|_{\infty,[0,T]}, \|D_{\theta\theta} g_1(\theta)\|_{\infty,[0,T]} \Big\}< \infty;
\]
$k_1$ is Lipschitz continuous with Lipschitz constant such that
\[
C_{k_1}:= \|k_1(\theta)\|_{\infty,[0,T]} \vee {\rm Lip}(k_1) < \infty.
\]
We can prove the following estimate (see the proof in the Appendix).

\begin{proposition}\label{roughGronwall}
	There exist a generic constant $P= P(b-a,\nu-\alpha)$ such that for all $0 \leq a \leq b\leq a+1$,
	\begin{eqnarray}\label{alphanorm1}
	&&\max \Big\{\ltn (\theta,\theta^\prime) \rtn_{x,2\alpha,[a,b]}, \ltn (\theta,\theta^\prime) \rtn_{x,2\alpha,[a,b]}^2, \ltn (\theta,\theta^\prime) \rtn_{x,2\alpha,[a,b]}^4 \Big\}  \notag\\
	&\leq&  P(b-a)M^{\frac{4}{\nu-\alpha}} \Big[1+ \Big(\ltn x \rtn_{\nu,[a,b]} +  \ltn \X \rtn_{2\nu,\Delta^2([a,b])} +  \ltn [x]\rtn_{2\nu,\Delta([a,b])}\Big)^\frac{8}{\nu-\alpha} \Big]
	\end{eqnarray} 
	where 
	\begin{eqnarray}\label{Mlinear}
	M := \max \Big\{C_{f_1}, C_{g_1}^2(1 + C_\alpha), C_{k_1}(1+ K_\alpha), C_{g_1} (C_\alpha + 1),\frac{1}{2}\Big \} 
	\end{eqnarray}
\end{proposition}

	{\bf Step 2.} It is now sufficient to estimate the quantity in \eqref{logx}. For any $0\leq a \leq t \leq 1$, rewrite \eqref{logx} in the integral form
	\allowdisplaybreaks
	\begin{eqnarray}\label{logx2}
	\log \|y_t\| 
	&=& \log \|y_a\| + \int_a^t \langle y_s, Ay_s + \frac{f(y_s)}{\|y_s\|}\rangle ds +  \int_a^t \langle \theta_s, C\theta_s \rangle dx_s  + \int_a^t \Big[\frac{1}{2}\|C\theta_s\|^2 - \langle \theta_s, C\theta_s \rangle^2\Big]d[x]_{0,s}\notag\\
	&\leq& \log \|y_a\| -\lambda (t-a) + \int_a^t h(\|y_s\|)ds + \Big\|\int_{a}^{t} \langle \theta_s, C\theta_s \rangle dx_s\Big\| \notag\\
	&& + \Big\|\int_{a}^{t} \Big[\frac{1}{2}\|C\theta_s\|^2 - \langle \theta_s, C\theta_s \rangle^2\Big]d[x]_{a,s}\Big\|. 
	\end{eqnarray}
	The last term in the last line of \eqref{logx2} can be estimated as
	\allowdisplaybreaks
	\begin{eqnarray}\label{logxyoung}
	&&\Big\|\int_a^b \Big[\frac{1}{2}\|C\theta_s\|^2 - \langle \theta_s, C\theta_s \rangle^2\Big]d[x]_{a,s} \Big\| \notag \\
	&& \leq \frac{3}{2} \|C\|^2 \Big|[x]_{a,b}\Big| + K_\alpha|b-a|^{3 \alpha} \ltn [x] \rtn_{2\alpha,\Delta^2([a,b])} \ltn\Big[\frac{1}{2}\|C\theta\|^2 - \langle \theta, C\theta\rangle^2\Big] \rtn_{\alpha,[a,b]}\notag\\
	&& \leq \frac{3}{2} \|C\| ^2 |b-a|^{2\alpha} \ltn [x] \rtn_{2\alpha,\Delta^2([a,b])}  + K_\alpha|b-a|^{3 \alpha} \ltn [x] \rtn_{2\alpha,\Delta^2([a,b])} \Big[\|C\|^2  + 4 \|C\|^2 \Big] \ltn \theta \rtn_{\alpha,[a,b]}\notag\\
	&& \leq \|C\|^2 |b-a|^{2\alpha}\ltn [x] \rtn_{2\alpha,\Delta^2([a,b])} \Big[ \frac{3}{2} + 5K_\alpha |b-a|^{\alpha} \Big(C_g \ltn x \rtn_\alpha + |b-a|^{2\nu-2\alpha}(\ltn x \rtn_\alpha+1) \ltn \theta,\theta^\prime \rtn_{x,2\alpha}\Big)\Big]\notag\\ 
	&& \leq \|C\|^2 |b-a|^{2\alpha}\ltn [x] \rtn_{2\alpha,\Delta^2([a,b])} \Big( \frac{3}{2} + 5K_\alpha C_G  |b-a|^{\alpha} \ltn x \rtn_\alpha \Big) \notag \\
	&& \qquad  + 5K_\alpha \|C\|^2 |b-a|\Big[ \frac{1}{2}\ltn [x] \rtn_{2\alpha,\Delta^2([a,b])}^2(\ltn x \rtn_\alpha+1)^2 + \frac{1}{2} \ltn \theta,\theta^\prime \rtn_{x,2\alpha}^2 \Big].
	\end{eqnarray}
Meanwhile the rough integral can be estimated as 
\allowdisplaybreaks
	\begin{eqnarray}\label{logxrough}
\Big\|\int_{a}^{b} \langle \theta_s, C\theta_s \rangle dx_s\Big\| 
	&\leq& \Big|\langle \theta_a, C\theta_a \rangle\Big| |x_b-x_a| + \Big|\langle \theta,C\theta \rangle^\prime_a \Big| |\X_{a,b}|\notag\\
	&& + C_\alpha |b-a|^{3\alpha} \Big(\ltn x \rtn_{\alpha,[a,b]} \ltn R^{\langle \theta,C\theta\rangle} \rtn_{2\alpha,[a,b]} + \ltn \langle  \theta,C\theta \rangle^\prime \rtn_{\alpha,[a,b]} \ltn \X \rtn_{2\alpha,\Delta^2([a,b])} \Big)\notag\\
	&\leq& \|C\| |b-a|^\alpha \ltn x \rtn_{\alpha,[a,b]} + 4 \|C\|^2 |b-a|^{2\alpha} \ltn \X \rtn_{2\alpha,\Delta^2([a,b])} \notag\\
	&& + C_\alpha |b-a|^{3\alpha} \Big(\ltn x \rtn_{\alpha,[a,b]} \ltn R^{\langle \theta,C\theta \rangle} \rtn_{2\alpha,[a,b]} + \ltn \langle  \theta,C\theta \rangle^\prime \rtn_{\alpha,[a,b]} \ltn \X \rtn_{2\alpha,\Delta^2([a,b])} \Big).\notag\\
	\end{eqnarray}
	To estimate the brackets of the last line of \eqref{logxrough}, we apply \eqref{yalpha} to get
	\begin{eqnarray*}
	\ltn \langle  \theta,C\theta \rangle^\prime \rtn_{\alpha,[a,b]} &\leq & \ltn \|C\theta\|^2 \rtn_{\alpha,[a,b]} +\ltn \langle \theta, C^2\theta\rangle \rtn_{\alpha,[a,b]} + \ltn \langle \theta, C\theta \rangle^2 \rtn_{\alpha,[a,b]}  + \ltn\langle \theta,C\theta \rangle \langle \theta, C\theta  \rangle \rtn_{\alpha,[a,b]} \\
	&\leq& 14 \|C\|^2 \ltn \theta \rtn_{\alpha,[a,b]}\notag \\
	&\leq&  14\|C\|^2 \Big(C_g \ltn x \rtn_\alpha + |b-a|^{2\nu-2\alpha}(\ltn x \rtn_\alpha+1) \ltn \theta,\theta^\prime \rtn_{x,2\alpha}\Big).
	\end{eqnarray*}
	Meanwhile
	\begin{eqnarray*}
		\|R^{\langle \theta,C\theta \rangle}_{s,t}\| 
		&\leq& \Big|\langle \theta_t,C\theta_t \rangle - \langle \theta_s,C\theta_s \rangle - \langle \theta_\cdot,C\theta\rangle^\prime_s x_{s,t}\Big| \\
		&\leq& 2 \|C\|\|R^\theta_{s,t}\| + 2 \|C\| \|\theta^\prime_s\|\|R^\theta_{s,t}\|\|x_{s,t}\| + \|C\| \|R^\theta_{s,t}\|^2 + \|C\| \|\theta^\prime_s\|^2\|x_{s,t}\|^2\\
		&\leq& 2 \|C\| \|R^\theta_{s,t}\| + 4 \|C\|^2\|R^\theta_{s,t}\|\|x_{s,t}\| + \|C\| \|R^\theta_{s,t}\|^2 + 4 \|C\|^3 \|x_{s,t}\|^2;
	\end{eqnarray*}
	thus it follows that
	\begin{eqnarray*}
	&& \ltn R^{\langle \theta,C\theta\rangle} \rtn_{2\alpha,[a,b]} \\
	& \leq& 2 \|C\| \ltn R^\theta \rtn_{2\alpha,[a,b]} + 4 \|C\|^2 |b-a|^{\alpha} \ltn x \rtn_{\alpha,[a,b]} \ltn R^\theta\rtn_{2\alpha,[a,b]} + \|C\| |b-a|^{2\alpha} \ltn R^\theta \rtn_{2\alpha,[a,b]}^2 \\
	&&+ 4 \|C\|^3 \ltn x \rtn_{\alpha,[a,b]}^2 \notag\\
	&\leq& 4 \|C\|^3 \ltn x \rtn_{\alpha,[a,b]}^2 + \Big( 2 \|C\| + 4 \|C\|^2 |b-a|^{\alpha} \ltn x \rtn_{\alpha,[a,b]} \Big) \ltn \theta,\theta^\prime\rtn_{x,2\alpha} +  \|C\| |b-a|^{2\alpha} \ltn \theta,\theta^\prime\rtn_{x,2\alpha}^2.
	\end{eqnarray*}
	Combining all the above estimates into \eqref{logxrough} and applying Cauchy inequality we get
	\begin{eqnarray}\label{logxrough1}
	\Big\|\int_{a}^{b} \langle \theta_s, C\theta_s\rangle dx_s\Big\|
	&\leq& \|C\| |b-a|^\alpha \ltn x \rtn_{\alpha,[a,b]} + 4 \|C\|^2 |b-a|^{2\alpha} \ltn \X \rtn_{2\alpha,\Delta^2([a,b])} \notag \\
	&& +  C_\alpha \|C\|^2|b-a|^{3\alpha} \Big( 4\|C\| \ltn x \rtn_{\alpha,[a,b]} \ltn x \rtn_{\alpha,[a,b]}^2  +  14 C_G \ltn \X \rtn_{2\alpha,\Delta^2([a,b])} \ltn x \rtn_{\alpha,[a,b]}  \Big)\notag\\
	&& + C_\alpha \|C\| |b-a|^{3\alpha} \Big\{\ltn x \rtn_{\alpha,[a,b]}^2 \Big(1+ 2 \|C\||b-a|^{\alpha} \ltn x \rtn_{\alpha,[a,b]} \Big)^2 \notag \\
	&&+ \ltn \theta,\theta^\prime\rtn_{x,2\alpha}^2  +  |b-a|^{2\alpha} \Big( \frac{1}{2}\ltn x \rtn_{\alpha,[a,b]}^2 + \frac{1}{2}\ltn \theta,\theta^\prime\rtn_{x,2\alpha}^4\Big) \notag \\
	&& + 7 \|C\||b-a|^{2\nu-2\alpha}\Big[(\ltn x \rtn_\alpha+1)^2 \ltn \X \rtn_{2\alpha,[a,b]}^2 +  \ltn \theta,\theta^\prime \rtn_{x,2\alpha}^2 \Big]\Big\}.
	\end{eqnarray}
	Replacing \eqref{logxyoung} and \eqref{logxrough1} into \eqref{logx2} using \eqref{alphanorm1} in Lemma \ref{roughGronwall}, we conclude that there exists an increasing polynomial with all positive coefficients  
	\[
	\kappa(t,a,x,\X,[x])=\kappa\Big(t-a,\ltn x \rtn_{\alpha,[a,t]}, \ltn \X \rtn_{2\alpha,\Delta^2([a,t])},\ltn [x] \rtn_{2\alpha,\Delta^2([a,t])}\Big), \quad \kappa(a,a,x,\X,[x]) =0, 
	\]
	and an increasing function $K: R^+ \to \R^+$ such that for all $0 \leq a \leq t \leq 1$
	\begin{equation}\label{logx3}
		\log \|y_t\| \leq \log \|y_a\| + \int_a^t \Big[h(\|y_s\|) + \|C\| K(\|y_s\|)-\lambda_A\Big] ds + \|C\| \kappa(t,a,x,\X,[x]),
	\end{equation}	
	which is similar to \eqref{logx2}. Because of \eqref{expect}, \eqref{gamma} holds for the realization $x$ and $\X$. Since $\lambda > k(0)$, we can choose $\|C\|<\epsilon$ small enough such that function $H(u) := h(u) + \|C\| K(u)$ is increasing function and $H(0) < \lambda_A$. Using \eqref{ysupest}, Lemma \ref{localstabPro} and Lemma \ref{globalstabPro}, we can then prove that system \eqref{RDE1} is locally/globally exponentially stable at zero for almost sure all the realization.
		
\end{proof}	

\begin{corollary}
	Let $\Phi(t,x,\X,[x])$ be the solution matrix of $dz_t = A z_t dt + C z_t dx_t$. Then there exists a function $\kappa(t,a,x,\X,[x])$ such that for any $\delta >0$
	\begin{equation}\label{phiest}
	\|\Phi(t,x,\X,[x])\| \leq \exp \Big\{-\lambda_A t + \|C\| \kappa(t,0,x,\X,[x]) \Big\}.
	\end{equation} 	
	As a result
	\begin{equation}\label{LEphi}
	\limsup \limits_{t \to \infty} \frac{1}{t}\log \|z_t\| \leq -\lambda_A + \|C\|\ \E\  \kappa(\delta,0,x,\X,[x]).
	\end{equation}
\end{corollary}

\begin{corollary}
	Consider the following system
	\begin{equation}\label{fSDE}
	dy_t = [Ay_t + f(y_t)]dt + Cy_tdB^H_t,\quad y_\cdot \in \R^d, 
	\end{equation}
	where $B^H$ is a fractional Brownian motion with Hurst index $\frac{1}{3}<H<1$; $A$ is negative definite and $f: \R^d \to \R^d$ is globally Lipschitz continuous, i.e. there exist contants $h_0, c_f >0$ such that 
	\begin{equation}\label{conditions}
	\langle y, Ay \rangle \leq - h_0 \|y\|^2,\quad \|f(y_1) - f(y_2)\| \leq c_f \|y_1-y_2\|, \quad \forall y_1,y_2\in \R^d. 
	\end{equation}
	Assume that $h_0 > c_f$. There exists an $\epsilon >0$ such that under condition $\|C\| < \epsilon$, $\varphi$ possesses a random pullback attractor consisting only one point $a(x)$, to which other random points converge to with exponential rate.
\end{corollary}
\begin{proof}
	The case $H>\frac{1}{2}$ is proved in \cite[Theorem 3.3]{duchongcong18}. For $\frac{1}{3}<H<\frac{1}{2}$, starting with the estimate \eqref{phiest}, we apply the H\"older inequality such that 
	\[
	\kappa(t,a,x,\X,[x]) \leq H_0 + (t-a) \tilde{\kappa} (t,a,x,\X,[x]), \quad \forall 0 \leq a \leq t \leq 1, 
	\]
	where $H_0>0$ is a constant and
	\[
	\tilde{\kappa}(t,a,x,\X,[x])=\tilde{\kappa}\Big(t-a,\ltn x \rtn_{\alpha,[a,t]}, \ltn [x] \rtn_{2\alpha,\Delta^2([a,t])}, \ltn \X \rtn_{2\alpha,\Delta^2([a,t])}\Big), \quad \tilde{\kappa}(a,a,x,\X) =0, 
	\]
	and $\tilde{\kappa}$ is an increasing function. It follows that $\Gamma(t,s,x,\X,[x]) = (t-s) \tilde{\kappa} (t,s,x,\X,[x])$ is a control function, and 
	\[
	\|\Phi(t,x,\X,[x])\| \leq \exp \Big\{\|C\| H_0 -\lambda_A t + \|C\| \Gamma(t,0,x,\X,[x]) \Big\}.
	\]
	The arguments are then similar to the proof of \cite[Theorem 4.4]{duchongcong18}. We stress here that for the rough case, it is proved in \cite{BRSch17} that the system \eqref{fSDE} generates a random dynamical system \cite{arnold}. 
\end{proof}	

\section{Appendix}

\begin{proof}[{\bf Proposition \ref{roughGronwall}}]
	Consider the solution mapping $\cM: \cD^{2\alpha}_x(\theta_a,g(\theta_a)) \to \cD^{2\alpha}_x(\theta_a,g(\theta_a)) $ defined by
	\[
	\cM (\theta,\theta^\prime)_t = (F(\theta,\theta^\prime)_t, g(\theta_t)),
	\]
	together with the seminorm
	\begin{eqnarray*}
		\ltn (\theta,\theta^\prime)\rtn_{x,2\alpha} = \ltn \theta^\prime \rtn_\alpha + \ltn R^\theta \rtn_{2\alpha},\quad
		\ltn \cM (\theta,\theta^\prime)\rtn_{x,2\alpha} = \ltn g_1(\theta) \rtn_\alpha + \ltn R^{F(\theta,\theta^\prime)} \rtn_{2\alpha}.
	\end{eqnarray*}
	We are going to estimate these seminorms. Observe from \eqref{nonlinRDE1} that $\theta^\prime = g_1(\theta_t)$, thus
	\allowdisplaybreaks
	\begin{eqnarray}
	\ltn \theta\rtn_\alpha &\leq&\ltn \theta^\prime \rtn_\infty \ltn x \rtn_\alpha + |T-a|^\alpha \ltn R^\theta\rtn_{2\alpha} \leq C_g \ltn x \rtn_\alpha + |b-a|^{\alpha} \ltn \theta,\theta^\prime \rtn_{x,2\alpha};\label{yalpha}\\	
	\ltn g_1(\theta) \rtn_\alpha &\leq& \ltn D_\theta g_1(\theta_\cdot) \rtn_\infty \ltn \theta\rtn_\alpha\leq C_{g_1} \ltn \theta \rtn_\alpha \label{gyalpha}.
	\end{eqnarray}
	Meanwhile using H\"older inequality
	\begin{eqnarray}\label{R}
	\|R^{F(\theta,\theta^\prime)}_{s,t}\|& \leq& \int_s^t \|f_1(u,\theta_u)\|du  + \|D_\theta g_1(\theta_s)g(\theta_s)\| |\X_{s,t}| +\|k_1(\theta_\cdot)\|_\infty |[x]_{s,t}| \notag\\
	&&+ K_\alpha |t-s|^{3\alpha} \ltn k_1(\theta) \rtn_\alpha \ltn [x] \rtn_{2\alpha}+ C_\alpha |t-s|^{3\alpha} \Big(\ltn x\rtn_\alpha \ltn R^{g_1(\theta)}\rtn_{2\alpha} + \ltn g_1(\theta)^\prime \rtn_\alpha \ltn \X \rtn_{2\alpha}\Big) \notag\\
	&\leq& |t-s|^{2\nu} \Big(\int_a^b \|f_1(u,\theta_u)\|^{\frac{1}{1-2\nu}}du\Big)^{1-2\nu} + C_{g_1}^2 |\X_{s,t}|+ C_k  |[x]_{s,t}| \notag\\
	&&+ K_\alpha |t-s|^{3\alpha} C_k\ltn \theta \rtn_\alpha \ltn [x] \rtn_{2\alpha} + C_\alpha |t-s|^{3\alpha} \Big(\ltn x\rtn_\alpha \ltn R^{g_1(\theta)}\rtn_{2\alpha} + \ltn g_1(\theta)^\prime \rtn_\alpha \ltn \X \rtn_{2\alpha}\Big),\notag\\
	\end{eqnarray}
	where we use the fact that $\theta^\prime = g_1(\theta)$ to get
	\begin{eqnarray*}
		\ltn g_1(\theta)^\prime \rtn_\alpha &=& \ltn D_\theta g_1(\theta)\theta^\prime \rtn_\alpha	\leq \ltn D_\theta g_1(\theta) \rtn_\infty \ltn \theta^\prime \rtn_\alpha + \ltn D_\theta g_1(\theta) \rtn_\alpha  \ltn \theta^\prime \rtn_\infty \\
		&\leq& C_{g_1} \ltn g(\theta)\rtn_\alpha + C_{g_1} \ltn \theta \rtn_\alpha \| g_1(\theta) \|_\infty \leq 2C_{g_1}^2 \ltn \theta \rtn_\alpha.
	\end{eqnarray*}
	On the other hand
	\begin{eqnarray*}
		\|R^{g_1(\theta)}_{s,t}\| &\leq& \int_0^1 \Big\|D_\theta g_1\Big(\theta_s + \eta (\theta_t-\theta_s)\Big) - D_\theta g_1(\theta_s)\Big\| \|\theta^\prime_s\| |x(t)-x(s)|d\eta  \\
		&& + \int_0^1 \Big\|D_\theta g_1\Big(\theta_s + \eta (\theta_t-\theta_s)\Big) \Big\| d \eta\ \|R^\theta_{s,t}\|, 
	\end{eqnarray*}
	thus
	\begin{eqnarray*}
		\ltn R^{g_1(\theta)}\rtn_{2\alpha} &\leq&  \|D_\theta g_1(\theta)\|_{\infty} \ltn R^\theta \rtn_{2\alpha} + \frac{1}{2} C_{g_1} \ltn g_1(\theta_s)\rtn_\infty \ltn x \rtn_\alpha \ltn \theta\rtn_\alpha \leq C_{g_1} \ltn R^\theta \rtn_{2\alpha} + \frac{1}{2}C_{g_1}^2 \ltn x \rtn_\alpha \ltn \theta \rtn_\alpha.	
	\end{eqnarray*}
	Combining these above estimates into \eqref{R}, we get for any $a<b\leq a+1$
	\begin{eqnarray*}
		\ltn R^{F(\theta,\theta^\prime)}\rtn_{2\alpha} &\leq& (b-a)^{2 \nu - 2\alpha}\Big(\int_a^b \|f_1(u,\theta_u)\|^{\frac{1}{1-2\nu}}du\Big)^{1-2\nu}+ C_{g_1}^2 \ltn\X\rtn_{2\alpha} + C_{k_1}\ltn [x]\rtn_{2\alpha} \\
		&& 
		+ K_\alpha C_{k_1} |b-a|^{\alpha} \ltn \theta \rtn_\alpha \ltn [x] \rtn_{2\alpha} \\
		&&+ C_\alpha |b-a|^{\alpha} \Big\{\ltn x\rtn_\alpha \Big[C_{g_1} \ltn R^\theta \rtn_{2\alpha} +\frac{1}{2}C_{g_1}^2 \ltn x \rtn_\alpha \ltn \theta \rtn_\alpha\Big]+  2 C_{g_1}^2 \ltn \theta \rtn_\alpha \ltn \X \rtn_{2\alpha}\Big\}\\
		&\leq& C_{f_1} (b-a)^{2\nu-2\alpha} + C_{g_1}^2\ltn\X\rtn_{2\alpha} + C_{k_1} \ltn [x]\rtn_{2\alpha}+ C_\alpha C_{g_1} (b-a)^\alpha\ltn x \rtn_\alpha  \ltn R^\theta \rtn_{2\alpha}\\
		&& + \Big\{ C_g^2C_\alpha \ltn \X \rtn_{2\alpha} + C_{k_1} K_\alpha \ltn [x] \rtn_{2\alpha} +\frac{1}{2} C_\alpha C_g^2 \ltn x \rtn_\alpha^2\Big\} |b-a|^\alpha\ltn \theta \rtn_\alpha.
	\end{eqnarray*}
 Together with \eqref{yalpha} and \eqref{gyalpha} we conclude that for any $a < b$ such that $b-a \leq1$ then 	
	\begin{eqnarray*}
		&&\ltn R^{F(\theta,\theta^\prime)}\rtn_{2\alpha} + \ltn g_1(\theta) \rtn_\alpha \\
		&\leq& C_{f_1} (b-a)^{2\nu-2\alpha} + C_{g_1}^2\ltn\X\rtn_{2\alpha} + C_{k_1} \ltn [x]\rtn_{2\alpha}+ C_\alpha C_{g_1} (b-a)^\alpha\ltn x \rtn_\alpha  \ltn \theta,\theta^\prime \rtn_{x,2\alpha}\\
		&& + \Big\{ \Big[C_{g_1}^2C_\alpha\ltn \X \rtn_{2\alpha} + C_{k_1} K_\alpha \ltn [x] \rtn_{2\alpha} + \frac{1}{2}C_\alpha C_{g_1}^2 \ltn x \rtn_\alpha^2\Big] |b-a|^\alpha + C_{g_1}\Big\}\times\\
		&&\qquad \qquad \qquad \times \Big[ C_{g_1} \ltn x \rtn_\alpha + |b-a|^{\alpha} \ltn \theta,\theta^\prime \rtn_{x,2\alpha} \Big] \\
		&\leq& C_{f_1} (b-a)^{2\nu-2\alpha} + C_{g_1}^2\ltn\X\rtn_{2\alpha} + C_{k_1} \ltn [x]\rtn_{2\alpha}+ C_{g_1}^2 \ltn x \rtn_\alpha\\
		&&+  \Big[C_{g_1}^2C_\alpha \ltn \X \rtn_{2\alpha} + C_{k_1} K_\alpha \ltn [x] \rtn_{2\alpha} + \frac{1}{2}C_\alpha C_{g_1}^2 \ltn x \rtn_\alpha^2\Big]C_{g_1} |b-a|^\alpha \ltn x \rtn_\alpha \\	
		&&+ \Big\{ \Big[C_{g_1}^2C_\alpha \ltn \X \rtn_{2\alpha} + C_{k_1} K_\alpha \ltn [x] \rtn_{2\alpha} + \frac{1}{2}C_\alpha C_{g_1}^2 \ltn x \rtn_\alpha^2\Big](b-a)^\alpha+ C_{g_1} + C_\alpha C_{g_1} \ltn x \rtn_\alpha\Big\}\times\\
		&& \qquad \qquad \qquad \times (b-a)^\alpha\ltn \theta,\theta^\prime \rtn_{x,2\alpha}\\
		&\leq& 	M \Big[|b-a|^{2\nu-2\alpha} + \ltn \X \rtn_{2\alpha} + \ltn [x] \rtn_{2\alpha} +\ltn x \rtn_\alpha +  \Big(\ltn \X \rtn_{2\alpha} + \ltn x \rtn_\alpha^2 +  \ltn [x] \rtn_{2\alpha}\Big)(b-a)^\alpha \ltn x \rtn_\alpha \Big] \\
		&& + M \Big\{ \Big(\ltn \X \rtn_{2\alpha} +  \ltn x \rtn_\alpha^2 +\ltn [x] \rtn_{2\alpha}\Big)\ltn x \rtn_\alpha +|b-a|^{2\nu -2\alpha}  +\ltn \X \rtn_{2\alpha} + \ltn [x] \rtn_{2\alpha}+ \ltn x \rtn_\alpha \Big\}	 \ltn \theta,\theta^\prime \rtn_{x,2\alpha}.
	\end{eqnarray*}
	Now construct for any fixed $\mu \in (0,1)$ a sequence of stopping times $\{\tau_k\}_{k \in \N}$ such that $\tau_0 = 0$ and
	\begin{equation}\label{stoppingtime}
	|\tau_{k+1} - \tau_k|^{2\nu-2\alpha} + |\tau_{k+1} - \tau_k|^{\nu-\alpha}\Big(\ltn x \rtn_{\nu,[\tau_k,\tau_{k+1}]} +  \ltn \X \rtn_{2\nu,\Delta^2([\tau_k,\tau_{k+1}])} +  \ltn [x] \rtn_{2\nu,\Delta^2([\tau_k,\tau_{k+1}])}\Big) = \frac{\mu}{2M}, 
	\end{equation}
	for all $k \in \N$, then it follows that 
	\begin{eqnarray*}
		\ltn x \rtn_\alpha &\leq& |\tau_{k+1} - \tau_k|^{\nu-\alpha} \ltn x \rtn_{\nu,[\tau_k,\tau_{k+1}]} < 1, \\
		\ltn \X \rtn_{2\alpha} &\leq& |\tau_{k+1} - \tau_k|^{2(\nu-\alpha)} \ltn \X \rtn_{\nu,\Delta^2([\tau_k,\tau_{k+1}])} < 1,\\ 
		\ltn [x] \rtn_{2\alpha} &\leq& |\tau_{k+1} - \tau_k|^{2(\nu-\alpha)} \ltn [x] \rtn_{\nu,\Delta^2([\tau_k,\tau_{k+1}])} < 1,
	\end{eqnarray*}
	hence it derives
	\begin{eqnarray*}
		&&\ltn g_1(\theta) \rtn_{\alpha,[\tau_k,\tau_{k+1}]} + \ltn R^{F(\theta,\theta^\prime)}\rtn_{2\alpha,[\tau_k,\tau_{k+1}]} \\
		&\leq& 2M \Big\{	|\tau_{k+1} - \tau_k|^{2\nu-2\alpha} + |\tau_{k+1} - \tau_k|^{\nu-\alpha}\Big(\ltn x \rtn_\nu +  \ltn \X \rtn_{2\nu} +  \ltn [x] \rtn_{2\nu}\Big)\Big\}(1+ \ltn \theta,\theta^\prime \rtn_{x,2\alpha})\\
		&\leq& \mu  + \mu  \ltn \theta,\theta^\prime \rtn_{x,2\alpha}. 
	\end{eqnarray*}
	Hence using the fact that $\theta^\prime = g_1(\theta)$ and $F(\theta,\theta^\prime) = \theta$ we conclude that
	\begin{equation}
	\ltn (\theta,\theta^\prime) \rtn_{x,2\alpha,[\tau_k,\tau_{k+1}]} \leq  \frac{\mu }{1-\mu}. 
	\end{equation}
	Therefore
	\[
	\ltn (\theta,\theta^\prime) \rtn_{x,2\alpha,[a,b]} \leq \frac{\mu}{1-\mu} N_{\frac{\mu}{2M},[a,b],\nu,\alpha}(\bx), 
	\]
	where $N_{\frac{\mu}{2M},[a,b],\nu,\alpha}(\bx)$ is the number of stopping times $\tau_k$ in the interval $[a,b]$. It is easy to see that
	\[
	b-a > N_{\frac{\mu}{2M},[a,b],\nu,\alpha}(\bx)\Big\{ \frac{\mu}{2M} \Big( 1 + \ltn x \rtn_{\nu,[a,b]} +  \ltn \X \rtn_{2\nu,\Delta^2([a,b])} +  \ltn [x] \rtn_{2\nu,\Delta^2([a,b])}\Big)^{-1} \Big\}^{\frac{1}{\nu-\alpha}}.
	\]
	All in all, we have just shown that for all $0\leq a \leq b \leq T$
	\begin{eqnarray}\label{alphanorm}
	&&\ltn (\theta,\theta^\prime) \rtn_{x,2\alpha,[a,b]} \notag\\
	&\leq&  \frac{b-a}{(1-\mu)\mu^{\frac{1}{\nu-\alpha}-1}} (2M)^{\frac{1}{\nu-\alpha}} \Big( 1 + \ltn x \rtn_{\nu,[a,b]} +  \ltn \X \rtn_{2\nu,\Delta^2([a,b])} +  \ltn [x]\rtn_{2\nu,\Delta^2([a,b])}\Big)^{\frac{1}{\nu-\alpha}}\notag \\
	&\leq&  \frac{(b-a)\mu}{2(1-\mu)} \Big(\frac{4M}{\mu}\Big)^{\frac{1}{\nu-\alpha}} \Big[1+ \Big(\ltn x \rtn_{\nu,[a,b]} +  \ltn \X \rtn_{2\nu,\Delta^2([a,b])} +  \ltn [x]\rtn_{2\nu,\Delta^2([a,b])}\Big)^\frac{1}{\nu-\alpha} \Big].
	\end{eqnarray}
	The other estimates for $\ltn (\theta,\theta^\prime) \rtn_{x,2\alpha,[a,b]}^2$ and $ \ltn (\theta,\theta^\prime) \rtn_{x,2\alpha,[a,b]}^4$ are direct consequences of Cauchy inequality for \eqref{alphanorm}.
	
\end{proof}

\section*{Acknowledgments}
This work was supported by the Max Planck Institute for Mathematics in the Science (MIS-Leipzig).

\end{document}